\documentclass[11pt,letterpaper]{amsart}

\usepackage{graphicx}
\usepackage{amsmath}
\usepackage{mathtools}
\usepackage{amssymb}
\usepackage{amsfonts}
\usepackage{amsthm}
\usepackage{cancel}
\usepackage[all]{xy}
\usepackage[neveradjust]{paralist}
\usepackage{enumitem}
\usepackage{multicol}
\usepackage{mathrsfs}
\usepackage{mathdots}
\usepackage[pagebackref,colorlinks=true,linkcolor=blue,citecolor=blue]{hyperref}
\usepackage{tikz-cd}
\usepackage{tikz}
\usepackage{blkarray}
\usepackage{scalerel,stackengine}

\UseRawInputEncoding

\usepackage{color}

\newcommand{\N}{\mathbb{N}}

\newcommand{\C}{\mathbb{C}}

\newcommand{\inv}{^{-1}}

\newcommand{\ol}{\overline}

\newcommand{\GL}{\mathrm{GL}}
\newcommand{\G}{\mathrm{G}}
\newcommand{\SO}{\mathrm{SO}}

\newcommand{\comment}[1]{}

\newtheorem{thm}{Theorem}[section]
\newtheorem{cor}[thm]{Corollary}
\newtheorem{lemma}[thm]{Lemma}
\newtheorem{prop}[thm]{Proposition}
\newtheorem {conj}[thm]{Conjecture}

\newtheorem {ques/conj}[thm]{Question/Conjecture}

\numberwithin{equation}{section}

\begin{document}

\title[Local Converse Theorem for quasi-split $\SO_{2l}$]{A local converse theorem for quasi-split even special orthogonal groups}

\author{Alexander Hazeltine}
\address{Department of Mathematics\\
University of Michigan\\
Ann Arbor, MI, 48109, USA}
\email{ahazelti@umich.edu}


\date{\today}

\dedicatory{}


\begin{abstract}
We give a direct proof of the local converse theorem for quasi-split non-split $\SO_{2l}$ over a local non-Archimedean field of characteristic $p\neq 2$, applying the 
theory of Howe vectors and partial Bessel functions. 
\end{abstract}

\maketitle


\section{Introduction}

Let $\G$ be a connected reductive group, $F$ be a non-Archimedean local field of characteristic zero, and let $G=\G(F)$. Converse theorems seek to uniquely identify representations of $G$ from its invariants. One such invariant is known as the local gamma factor. These factors are arithmetic in nature and play very important roles in the Langlands program. 
More precisely, let $\pi$ be an irreducible generic representation of $G$. The family of local twisted gamma factors is denoted by $\gamma(s, \pi \times \tau, \psi)$, where $\tau$ is an irreducible generic representation of $\GL_n(F)$, $\psi$ an additive character of $F,$ and $s\in\mathbb{C}$. These can be defined using Rankin--Selberg convolution (e.g., \cite{GRS98, JPSS83, Kap10, Kap12, Kap13a, Kap15, Sou93, ST15}) or the Langlands--Shahidi method (\cite{Sha84, Sha90}). For $\GL_l,$ Jacquet conjectured the following.

\begin{conj}[Local Converse Theorem for $\GL_l$]\label{lcp}
Let $\pi,\pi'$ be irreducible generic representations
of $\GL_l(F)$. Suppose that they have the same central character.
If
\[
\gamma(s, \pi \times \tau, \psi) = \gamma(s, \pi' \times \tau, \psi),
\]
for all irreducible
generic representations $\tau$ of $\GL_n(F)$ with $1 \leq n \leq 
\lfloor\frac{l}{2}\rfloor$, then $\pi \cong \pi'$.
\end{conj}

Conjecture \ref{lcp} was proved by Chai (\cite{Cha19}), and by Jacquet and Liu (\cite{JL18a}), independently, using different analytic methods. It is known as the local converse theorem for $\GL_l.$ For general quasi-split reductive groups over non-Archimedean local fields, Jiang gave a general conjecture for local converse theorems for generic representations in \cite[Conjecture 3.7]{Jia06} and many cases have since been established. A non-exhaustive list of such results includes  $\SO_{2l+1}$ (\cite{JS03, Jo22}), $\mathrm{Sp}_{2l}$ (\cite{Jo22, Zha18}), $\mathrm{U}(l,l)$ (\cite{Mor18, Zha18}), $\mathrm{U}_{2l+1}$ (\cite{Zha19}), split $\SO_{2l}$ (\cite{HKK23, HL22b}), and $\widetilde{\mathrm{Sp}}_{2l}$ (\cite{Haa22}). 

We note that, by applying Langlands functoriality and Conjecture \ref{lcp}, several of the above cases would be implied by the works of Arthur and Mok (\cite{Art13, Mok15}). However, it is desirable to have direct proofs which reflect the intrinsic properties of the groups as is carried out in \cite{HL22b, Jo22, Zha18, Zha19}. We remark more on this point later.

The main goal of this article is to establish the local converse theorem for the quasi-split non-split group $\SO_{2l}$ directly. We note that the case for split and quasi-split non-split $\SO_{2l}$ was also proved independently in \cite{HKK23} using the theta correspondence (similarly to \cite{Haa22} but using more refined arguments). However, the approaches of this paper (and \cite{HL22b}) contrast significantly with that of \cite{HKK23}. In particular, we utilize the properties of partial Bessel functions to compute the zeta integrals directly.

As in the split case, one difficulty in the local converse theorem for quasi-split non-split $\SO_{2l}$ is the existence of the outer automorphism. We show that the twisted local $\gamma$-factors determine generic representations uniquely up to conjugation by the outer automorphism. 
More precisely, in this paper, we prove the following theorem. 

\begin{thm}[The Local Converse Theorem for quasi-split non-split  $\SO_{2l}$, supercuspidal case]\label{converse thm intro} Let $F$ be a non-Archimedean field of characteristic $p\neq 2$ and $\pi$ and $\pi^\prime$ be irreducible $\psi$-generic supercuspidal representations of quasi-split non-split $\SO_{2l}(F)$ with the same central character $\omega$. If $$\gamma(s, \pi\times\tau,\psi)=\gamma(s, \pi^\prime\times\tau,\psi),$$ 
for all irreducible generic representations $\tau$ of $\GL_n(F)$ with $n\leq l,$
then $\pi\cong\pi'$ or $\pi\cong \pi'^c,$
 where $c$ is the outer automorphism. 
\end{thm}

We note that the above theorem is equivalent to an analogous local converse theorem using the normalized twisted local $\gamma$-factor defined by Kaplan in \cite{Kap13a} (see the discussion at the end of \S\ref{zeta integrals and gamma factor}). By using the multiplicative properties for the normalized twisted local $\gamma$-factors (\cite{Kap13a}), Theorem \ref{converse thm intro} implies the following theorem by similar arguments as in \cite[\S 3.2]{JS03}. We omit the proof.

\begin{thm}[The Local Converse Theorem for quasi-split non-split $\SO_{2l}$, generic case]\label{converse thm intro generic} Let $F$ be a non-Archimedean field of characteristic $0$ and $\pi$ and $\pi^\prime$ be irreducible $\psi$-generic representations of quasi-split non-split $\SO_{2l}(F)$ with the same central character $\omega$. If $$\gamma(s, \pi\times\tau,\psi)=\gamma(s, \pi^\prime\times\tau,\psi),$$ 
for all irreducible generic representations $\tau$ of $\GL_n(F)$ with $n\leq l,$
then $\pi\cong\pi'$ or $\pi\cong \pi'^c,$
 where $c$ is the outer automorphism. 
\end{thm}

We remark that in the case that $F$ has characteristic $p\neq 2,$ Jo showed the analogue of Theorem \ref{converse thm intro} implies Theorem \ref{converse thm intro generic} for symplectic and odd special orthogonal groups (\cite{Jo22}) and a similar argument works in the split case of even special orthogonal groups (\cite{HL22b}). However, it seems that not all of the preliminary work required for such an argument has been carried out in the case of quasi-split special orthogonal groups yet. Hence, we require $F$ to have characteristic $0$ in Theorem \ref{converse thm intro generic}.

Theorems \ref{converse thm intro} and \ref{converse thm intro generic} suggest that local twisted gamma factors may not be able to distinguish the irreducible generic supercuspidal representations $\pi$ and $\pi^c$ of $\SO_{2l}(F).$ This is in fact true (see Corollaries \ref{conj gamma1}, \ref{conj gamma2} and \ref{conj gamma3}). This is a unique phenomenon for both split and quasi-split non-split even special orthogonal groups among all the classical groups. This is consistent with the work of Arthur on the local Langlands correspondence and local Langlands functoriality (\cite{Art13}), and the work of Jiang and Soudry on local descent for $\SO_{2l}$ (\cite{JS12}).

Now, we briefly introduce the ideas used to prove Theorem \ref{converse thm intro}. 
As in many other proven cases of classical groups (\cite{HL22b, Jo22, Zha18, Zha19}), we utilize the theory of Howe vectors and  partial Bessel functions of $\pi.$ These are particular Whittaker functions in the Whittaker model of $\pi$. Let $U_{\SO_{2l}}$ be the set of upper triangular unipotent elements in $\SO_{2l}$ and fix a generic character $\psi$ on $U_{\SO_{2l}}.$ Denote the Whittaker model of $\pi$ by $\mathcal{W}(\pi,\psi).$ For $m\in\mathbb{N}$, we define certain compact subgroups $H_m$ of $\SO_{2l}$ and characters $\psi_m$ of $H_m$ such that $U_{\SO_{2l}}=\cup_{m\geq1}(H_m\cap U_{\SO_{2l}})$ and $\psi_m(u)=\psi(u)$ for any $u\in H_m\cap U_{\SO_{2l}}$ (see \S\ref{Howe vectors}). A partial Bessel function for $\pi$ is a function $W_m\in\mathcal{W}(\pi,\psi)$ such that $W_m(I_{2l})=1$ and
$$
W_m(ugh)=\psi(u)\psi_m(h)W_m(g),
$$
for any $u\in U_{\SO_{2l}}$, $h\in H_m$, and $g\in\SO_{2l}.$

Let $\mathcal{M}(\pi)$ be the set of matrix coefficients of $\pi$. 
Let $C^\infty(\SO_{2l},\omega,\psi)$ be the space of smooth compactly supported functions $W$ satisfying $W(zg)=\omega(z)W(g)$ and $W(ug)=\psi(u)W(g)$ for any $z\in Z, u\in U_{\SO_{2l}},$ and $g\in\SO_{2l}(F).$ Since $\pi$ is generic, we have a nonzero map
$$f \in \mathcal{M}(\pi) \mapsto W^f \in C^\infty(\SO_{2l},\omega,\psi),$$ given by 
$$
W^f(g)=\int_{U_{\SO_{2l}}} \psi\inv(u)f(ug)du.
$$
By choosing
$f\in\mathcal{M}(\pi)$ such that $W^f(I_{2l})=1$ and letting $m$ be a large enough positive integer, we can associate a Howe vector to $W^f$, denoted by 
$B_m(g;f)$, $g \in \SO_{2l}(F)$, which is a partial Bessel function (see \S \ref{Howe vectors}).

Fix matrix coefficients $f\in\mathcal{M}(\pi)$ and $f'\in\mathcal{M}(\pi')$ such that $W^f(I_{2l})=W^{f'}(I_{2l})=1.$ The goal is to study the partial Bessel functions $B_m(g,f)$ and $B_m(g,f')$ under the assumption of the equality of local twisted gamma factors. First, we study the support of the partial Bessel functions and partition it based on the Bruhat cells corresponding to Weyl elements in the sets $\mathrm{B}_n(\SO_{2l})$ for $n=1,\dots,l-1$ (see Corollary \ref{Besselpartnonsplit}). The work of Cogdell, Shahidi, and Tsai (\cite{CST17}) then gives that for $1 \leq i \leq l-1$, there exist functions $f_i \in C_c^\infty(\SO_{2l},\omega)$ with support lying in the Bruhat cells corresponding to $\mathrm{B}_i(\SO_{2l}),$ such that 
$$
B_m(g,f)-B_m(g,f')=\sum_{n=1}^{l-1} B_m(g,f_n),
$$
for any $g\in \SO_{2l}$ (see Corollary \ref{C(0)}).
In Theorem \ref{l-2 theorem}, we show that for $1 \leq k \leq l-2$, equality of the local twisted gamma factors up to $\GL_k$ implies that 
$$
B_m(g,f)-B_m(g,f')=\sum_{n=k+1}^{l-1} B_m(g,f_i).
$$
Hence, equality of the local twisted gamma factors up to $\GL_{l-2}$ implies that 
$B_m(g,f)-B_m(g,f')=B_m(g,f_{l-1}).$ Up until this point, the results parallel those in \cite{Jo22, Zha18, Zha19}. Equality of the local twisted gamma factors up to $\GL_{l-1}$ only gives that $B_m(g,f)-B_m(g,f')=B_m(g,f_{l-1})=0$ if $cgc=g$ (see Theorem \ref{l-1 thm}). This is different from the cases in \cite{Jo22, Zha18, Zha19}, but is expected as it parallels the split case of $\SO_{2l}$ in \cite{HL22b}. Hence, it remains to determine the set elements $g$ such that $cgc\neq g.$ In the split case of $\SO_{2l},$ equality of the local twisted gamma factors up to $\GL_l$ implied that $B_m(g,f)-B_m(g,f')+B_m^c(g,f)-B_m^c(g,f')=0$ where $B_m^c(g,f)$ and $B_m^c(g,f')$ is are partial Bessel functions of $\pi^c$ and $\pi'^c,$ respectively (\cite[Theorem 7.9]{HL22b}). We obtain an analogous result for quasi-split non-split $\SO_{2l}.$

\begin{thm}[Theorem \ref{l thm}]\label{l thm intro}
Let $\pi$ and $\pi^\prime$ be irreducible supercuspidal $\psi$-generic representations of quasi-split non-split $\SO_{2l}(F)$ with the same central character. Fix matrix coefficients $f\in\mathcal{M}(\pi)$ and $f'\in\mathcal{M}(\pi')$ such that $W^f(I_{2l})=W^{f'}(I_{2l})=1.$ If $$\gamma(s, \pi\times\tau,\psi)=\gamma(s, \pi^\prime\times\tau,\psi),$$ as functions of the complex variable $s$,
for all irreducible generic representations $\tau$ of $\GL_n$ with $1 \leq n\leq l,$
then we have that
$$B_m(g,f)-B_m(g,f')+B_m^c(g,f)-B_m^c(g,f')=0,$$
for any $g\in\SO_{2l}( F)$.
\end{thm}

By the uniqueness of Whittaker models, it follows  that $\pi\cong\pi'$ or $\pi\cong \pi'^c$ and hence Theorem \ref{converse thm intro} is proved. The proof of Theorem \ref{converse thm intro} is outlined in more detail in \S \ref{outline of the proof} and the full details are provided in \S \ref{proof of the local converse theorem}.

An analogous version of Theorem \ref{l thm intro} is proven in the case of split $\SO_{2l}$ in \cite[Theorem 1.4]{HL22b}. We discuss the differences in the proofs. In the split case, the theory of Howe vectors and partial Bessel functions was already developed in \cite{Bar95}. Unfortunately, the quasi-split case was not considered and so we introduce it in \S\ref{Partial Bessel functions}.
The fact that the maximal torus is not split also requires new ideas to overcome. For example, many arguments in the split case utilize the isomorphism between $F$ and the simple root subgroups of $U_{\SO_{2l}}.$ In the quasi-split non-split case, we must be more careful. Instead we utilize certain elements of $U_{\SO_{2l}}$ and carefully apply the properties of partial Bessel functions to obtain results that are similar to the split case. There is one particular difference which makes the proof of Theorem 1.4 simpler than the split case. Namely, all of the Weyl elements in the Bessel support for quasi-split non-split $\SO_{2l}$ are self-conjugate (under the outer conjugation). In the split case, there are Weyl elements which are not self-conjugate and this significantly increases the amount of detail needed to prove \cite[Theorem 1.4]{HL22b}.

As a direct application of Theorem \ref{converse thm intro generic}, following the same argument as in \cite[Theorem 5.3]{JS03}, we also obtain the following weak rigidity theorem for irreducible generic cuspidal automorphic representations of quasi-split non-split $\SO_{2l}(\mathbb{A})$, where $\mathbb{A}$ is the ring of adeles for a number field (see also \cite[Theorem 1.5]{HL22b} for the split case). The proof is omitted.

\begin{thm}\label{rigidity}
Let $\Pi =\otimes_v \Pi_v$ and $\Pi'=\otimes_v \Pi'_v$ be two irreducible generic cuspidal automorphic representations of the quasi-split non-split group $\SO_{2l}(\mathbb{A})$. If $\Pi_v \cong \Pi'_v$ for almost all places $v$, then $\Pi_v \cong \Pi'_v$ or $\Pi_v \cong \Pi'^c_v$ for every place $v.$ 
\end{thm}

Theorem \ref{converse thm intro generic} and hence Theorem \ref{rigidity} were obtained independently by Haan-Kim-Kwon (\cite{HKK23}), using the different method of theta correspondence. Also we remark that Theorem \ref{converse thm intro generic} can also be obtained by applying the theory of local Langlands functoriality and the work of Arthur (\cite{Art13}).
 We note that our method is independent of the work of Arthur (\cite{Art13}). The advantage of our method is that it reflects the intrinsic properties of the partial Bessel functions of special even orthogonal groups.

As mentioned earlier, direct proofs of the local converse theorems for split and quasi-split non-split even special orthogonal groups are desirable. The arguments give more detailed information about the partial Bessel functions gained from equality of the twisted local $\gamma$-factors for specific twists. In particular, it is an open problem to determine a set of invariants which determine the representations of $\SO_{2l}$ uniquely (not just up to outer conjugation). Recently, in the case of split $\SO_4$, Yan and Zhang used the theory of Howe vectors and the results of \cite{HKK23, HL22b} along with work on a twisted exterior square local $\gamma$-factor to determine the representations uniquely (\cite[Theorem 1.1]{YZ23}). However, twisted exterior square local $\gamma$-factors for $\SO_6$ are not enough to distinguish $\pi$ and $\pi^c$ (see the discussion at the end of \cite[\S1]{YZ23} or \cite[Proposition 7.3]{Mat24}).

We remark that there has been extensive parallel work on converse theorems for finite groups. Over finite fields, the converse theorem for generic cuspidal representations has an analogous form to the local case; however, the arguments tend to be simpler since one can work with Bessel functions as opposed to partial Bessel functions. Let $\mathbb{F}_q$ be the finite field of $q$ elements, where $q$ is some power of an odd prime. For $\GL_l(\mathbb{F}_q)$, Nien proved the converse theorem (\cite{Nie14}) by using normalized Bessel functions and the twisted gamma factors defined by Roditty (\cite{Rod10}). In \cite{LZ22a}, Liu and Zhang defined the twisted gamma factors for $\mathrm{Sp}_{2l}(\mathbb{F}_q)$, $\SO_{2l+1}(\mathbb{F}_q)$, $\mathrm{U}_{2l}(\mathbb{F}_q)$, and $\mathrm{U}_{2l+1}(\mathbb{F}_q)$ and proved the corresponding converse theorems. Liu and Zhang also established the converse theorem for $\mathrm{G}_2(\mathbb{F}_q)$ (\cite{LZ22b}). Liu and the author proved the converse theorem for split $\SO_{2l}(\mathbb{F}_q)$ in \cite{HL22a}. The case of quasi-split non-split $\SO_{2l}(\mathbb{F}_q)$ was recently proven by the author in \cite{Haz23a} and served as a guide for the proof of the local converse theorem in this paper. We remark that in addition to the above differences, the local case of even special orthogonal groups (both split or quasi-split) have an additional technical difference when compared to their finite field counterparts or their analogues for other classical groups over local fields. Specifically, a significant amount of work is needed to guarantee the smoothness of a certain function (see the proof of Theorem \ref{l thm} or \cite[Theorem 6.13]{HL22b}).

Here is the structure of this paper. In \S \ref{gps and reps}, we introduce the groups and representations considered in this paper. 
In \S \ref{zeta integrals and gamma factor}, we recall the definitions of the zeta integrals and twisted $\gamma$-factors. In \S \ref{Partial Bessel functions}, we introduce the theory of Howe vectors and partial Bessel functions and outline the proof of Theorem \ref{converse thm intro}. In
\S \ref{sections}, we construct a section which is used in calculating the zeta integrals. In \S \ref{proof of the local converse theorem}, we study the $\GL_n$ twists, $1 \leq n \leq l$, and show the relation between the $\GL_n$ twists and the partial Bessel functions (Theorems \ref{l-2 theorem}, \ref{l-1 thm}, \ref{l thm}). Then, we prove Theorem \ref{l thm intro} and deduce Theorem \ref{converse thm intro}. 

\subsection*{Acknowledgements}

The author would like to thank Professor Dihua Jiang, Professor Freydoon Shahidi, and Professor Baiying Liu for their interest and constant support.
The author would also like to thank Professor Qing Zhang for helpful communications and comments. 

\section{The groups and representations}\label{gps and reps}

Let $n, l \in \N$ and $F$ be a non-Archimedean local field of characteristic $0$. We fix a nontrivial unramified additive character $\psi$ of $F$ and let $\mathfrak{o}$  be the integers of $F$ with maximal ideal $\mathfrak{p}$ and uniformizer $\varpi.$

Let $\GL_n$ be the group of matrices with entries in $F$ and non-zero determinant. Let $I_n$ be the identity element and fix

$$J_n=\left(\begin{matrix}
0 & 0 & \cdots & 0 & 1 \\
0 & 0 & \cdots & 1 & 0 \\
\vdots & \vdots & \iddots & \vdots & \vdots \\
0 & 1 & \cdots & 0 & 0 \\
1 & 0 & \cdots & 0 & 0 \\
\end{matrix}\right).$$ For $\rho\in F^\times$, we define
$$J_{2l,\rho}=\mathrm{diag}(I_{l-1}, \left(\begin{matrix}
0 & 1 \\
-\rho & 0
\end{matrix}\right), I_{l-1}) \cdot J_{2l}.$$ We set $\SO_{2n+1}=\{g\in\GL_{2n+1} \, | \, \mathrm{det}(g)=1, {}^tgJ_{2n+1} g = J_{2n+1}\}$ and $\SO_{2l}(J_{2l,\rho})=\{g\in\GL_{2l} \, | \, \mathrm{det}(g)=1, {}^tgJ_{2l,\rho} g = J_{2l,\rho}\}$. We often simply write this as $\SO_{2l}.$ The group $\SO_{2l}$ is split if $\rho\in F^2$ and is quasi-split, but not split, otherwise (in this case $\SO_{2l}$ splits over the quadratic extension $F(\sqrt{\rho})$). Unless otherwise explicitly stated, hereinafter, we fix $\SO_{2l}$ to be the non-split group $\SO_{2l}(J_{2l,\rho})$ for a fixed $\rho\notin F^2.$ 
Let $U_{\GL_{n}}$ and $U_{\SO_{2l}}$ be the subgroups of upper triangular matrices in $\GL_n$ and $\SO_{2l}$, respectively. Fix $B_{\SO_{2l}}=T_{\SO_{2l}}U_{\SO_{2l}}$ to be the standard Borel subgroup of $\SO_{2l}$ where $T_{\SO_{2l}}$ consists of all elements of the form
$$
t=\mathrm{diag}(t_1,\dots,t_{l-1},\left(\begin{matrix}
a & b\rho \\
b & a
\end{matrix}\right),t_{l-1}\inv,\dots,t_1\inv),
$$
where $t_1,\dots,t_{l-1}\in F^\times$ and $a,b\in F$ with  $a^2-b^2\rho=1.$ $T_{\SO_{2l}}$ is a maximal torus in $\SO_{2l}.$

Set $$
c=\mathrm{diag}(I_{l-1},
1,-1, I_{l-1}).$$ Note that $c\notin \SO_{2l};$ however, $c \SO_{2l} c\inv=c\SO_{2l}c= \SO_{2l}$. Given a representation $\pi$ of $\SO_{2l},$ we define a new representation $\pi^c$ of $\SO_{2l}$ by $\pi^c(g)=\pi(cgc).$

We now discuss the embeddings needed to define the zeta integrals later (see \cite{Kap13a, Kap15} for discussions on these embeddings). If $n<l$ we embed $\SO_{2n+1}$ into $\SO_{2l}$ via
$$
\left(
\begin{matrix}
A & B & C \\
D & E & K \\
L & P & Q 
\end{matrix} 
\right)
\mapsto
\mathrm{diag}(I_{l-n-1},
M^{-1}\left(\begin{matrix}
A & & B & C \\
 & 1 & & \\
D & & E & K \\
L & & P & Q 
\end{matrix}\right)
M, I_{l-n-1}),
$$
where $A$ and $Q$ are $n\times n$ matrices and 
$$
M =\mathrm{diag}(I_n,\left(\begin{matrix}
    0 & 2 \\
    1 & 0
    \end{matrix}\right), I_n). 
$$
This embeds $\SO_{2n+1}$ into the standard Levi subgroup $\GL_{l-n-1}\times\SO_{2n+2}$ of $\SO_{2l}$.

Set $\gamma=\frac{\rho}{2}.$ If $l=n$, we embed $\SO_{2l}$ into $\SO_{2l+1}$ via
$$
\left(\begin{matrix}
A & B \\
C & D
\end{matrix}\right)
\mapsto
M^{-1}
\left(\begin{matrix}
A & & B \\
 & 1 & \\
C & & D  
\end{matrix}\right)
M,
$$
where $A, B, C,$ and $D$ are $l\times l$ matrices and
$$
M=\mathrm{diag}(I_{l-1},\left(\begin{matrix}
    0 & 1 & 0 \\
    \frac{1}{2} & 0 & \frac{1}{2\gamma} \\
    \frac{1}{2} & 0 & \frac{-1}{2\gamma} 
    \end{matrix}\right), I_{l-1}).
$$
Note that the embedding takes the torus $T_{\SO_{2l}}$ to a torus in $\SO_{2l+1}$, but not the standard torus. Indeed, the embedding takes $t$ (as above) to
$$
\mathrm{diag}(s,\left(\begin{matrix}
\frac{1}{2}(1+a) & b & \frac{1}{2\gamma}(1-a) \\
\gamma b & a & -b \\
\frac{\gamma}{2}(1-a) & -\gamma b & \frac{1}{2}(1+a)
\end{matrix}\right), s^*)
$$
where $s=\mathrm{diag}(t_1, t_2, \dots, t_{l-1})$.

Next we discuss the representations. Recall that $U_{\GL_n}$ and $U_{\SO_{2l}}$ are the subgroups of upper triangular matrices in $\GL_n$ and $\SO_{2l}$ respectively and that we fixed an additive nontrivial character $\psi$ of $F$. We abuse notation and define a generic character, which we will also call $\psi$, on $U_{\GL_n}$ and $U_{\SO_{2l}}$. For $u=(u_{i,j})_{i,j=1}^n\in U_{\GL_n}$, we set $\psi(u)=\psi\left(\sum_{i=1}^{l-1} u_{i,i+1}\right).$ For $u=(u_{i,j})_{i,j=1}^{2l}\in U_{\SO_{2l}},$ we set 
$$
\psi(u)= \psi\left(\sum_{i=1}^{l-2} u_{i,i+1}+\frac{1}{2}u_{l-1,l+1}  \right).
$$

We say an irreducible representation $\pi$ of $\SO_{2l}$ is $\psi$-generic if $$\mathrm{Hom}_{U_{\SO_{2l}}}(\pi,\psi)\neq 0.$$
Similarly, we say an irreducible representation $\tau$ of $\GL_{n}$ is $\psi$-generic if $$\mathrm{Hom}_{U_{\GL_{n}}}(\tau,\psi)\neq 0.$$
A nonzero intertwining operator lying in these spaces is called a Whittaker functional and it is well known that Whittaker functionals are unique up to scalars. Fix $\Gamma\in\mathrm{Hom}_{U_{\SO_{2l}}}(\pi,\psi)$ to be a nonzero Whittaker functional. For $v\in\pi$, let $W_v(g)=\Gamma(\pi(g)v)$ for any $g\in\SO_{2l}$ and set $\mathcal{W}(\pi,\psi)=\{W_v \, | \, v\in\pi\}.$ $\mathcal{W}(\pi,\psi)$ is called the $\psi$-Whittaker model of $\pi.$ By Frobenius reciprocity, $\mathrm{Hom}_{U_{\SO_{2l}}}(\pi,\psi)\cong\mathrm{Hom}_{SO_{2l}}(\pi,\mathrm{Ind}_{U_{\SO_{2l}}}^{\SO_{2l}}(\psi)).$ Thus, $\pi$ can be realized as a subrepresentation of $\mathrm{Ind}_{U_{\SO_{2l}}}^{\SO_{2l}}(\psi)$ via the map $\pi \rightarrow \mathcal{W}(\pi,\psi)$ given by $v\mapsto W_v.$ Moreover, by the uniqueness of Whittaker models, this subrepresentation occurs with multiplicity one inside $\mathrm{Ind}_{U_{\SO_{2l}}}^{\SO_{2l}}(\psi)$. We also note that the analogous results hold for $\psi$-generic representations $\tau$ of $\GL_n.$

Let  $Q_n=L_n V_n$ be the standard Siegel parabolic subgroup of $\SO_{2n+1}$ with Levi subgroup $L_n\cong \GL_n.$ For $a\in\GL_n$ we let $l_n(a)=\mathrm{diag}(a,1,a^*)\in L_n$ where $a^*=J_n{}^ta^{-1}J_n.$ Let $\tau$ be an irreducible $\psi^{-1}$-generic representation of $\GL_{n}$, $s\in\C$, and set $I(\tau, s)=\mathrm{Ind}_{Q_n}^{\SO_{2n+1}}(\tau |\mathrm{det}|^{s-\frac{1}{2}}).$ An element $\xi_s\in I(\tau,s)$ is a function $\xi_s:\SO_{2n+1}\rightarrow\tau$ satisfying $$
\xi_s(l_n(a)ug)=|\mathrm{det}(a)|^{s-\frac{1}{2}}\tau(a)(\xi_s(g)), \forall a\in\GL_n, u\in V_n, g\in\SO_{2n+1},
$$ and is right translation invariant by some compact open subgroup.
Fix a nonzero homomorphism $\Lambda_\tau \in \mathrm{Hom}_{U_{\GL_{n}}}(\tau,\psi^{-1})$. For $\xi_s\in I (\tau,s)$, we let $f_{\xi_s} : \SO_{2n+1}\times\GL_n \rightarrow \C$ be the function given by $$
f_{\xi_s}(g,a)=\Lambda_\tau(\tau(a)\xi_s(g)).
$$
Let $I(\tau,s,\psi^{-1})$ be the space of functions generated by $f_{\xi_s}, \xi_s\in I(\tau,s).$ Note that for $f_s\in I(\tau,s,\psi^{-1}),$ we have 
$$
f_s(g,ua)=\psi^{-1}(u)f_s(g,a), \forall g\in\SO_{2n+1}, u\in U_{\GL_n}, a\in\GL_n.
$$
We  also let $\tau^*$ be the contragredient representation of $\GL_n$ defined by $\tau^*(a)=\tau(a^*).$

\section{Zeta integrals and the gamma factor}\label{zeta integrals and gamma factor}

In this section, we recall the definition of the zeta integrals in \cite{Kap13a}. In the case $n=l$, we already have the notation needed; however, we need to define another unipotent subgroup and character when $n<l.$ 

Suppose that $n<l.$ Let $P_{l-n-1}=M_{l-n-1}N_{l-n-1}$ be the standard parabolic subgroup of $\SO_{2l}$ with Levi subgroup, $M_{l-n-1}$, isomorphic to $\GL_{l-n-1}\times\SO_{2n+2}.$ We embed $U_{\GL_{l-n-1}}$ inside $\GL_{l-n-1}$ which is realized in $M_{l-n-1}.$ Then we define the unipotent subgroup $N^{l-n}=U_{\GL_{l-n-1}}N_{l-n-1}.$ That is,

$$
N^{l-n}=\left\{\left(\begin{matrix}
u_1 & v_1 & v_2 \\
 & I_{\SO_{2n+2}} & v_1^\prime \\
 & & u_1^*
\end{matrix}\right) \in \SO_{2l} \, | \, u_1\in U_{\GL_{l-n-1}} \right\}.
$$
For $v=(v_{i,j})\in N^{l-n}$ we define a character, which we also call $\psi,$ of $N^{l-n}$ by

$$
\psi(v)= 
   \psi\left(\sum_{i=1}^{l-n-1} v_{i,i+1}+\frac{1}{4}v_{l-n-1,l}-\frac{1}{2} v_{l-n-1,l+1}   \right) .
$$
Note that this character is trivial when $n=l-1.$ Let $H=\SO_{2n+1}N^{l-n}$ where $\SO_{2n+1}$ is realized via the embedding into $\SO_{2n+2}$ inside $M_{l-n-1}$ and extend $\psi$ trivially across $\SO_{2n+1}$ so that $\psi$ is a character of $H$.

Let $\pi$ be an irreducible $\psi$-generic supercuspidal representation of $\SO_{2l}$, $\tau$ be a $\psi\inv$-generic representation of $\GL_n,$  $W\in\mathcal{W}(\pi,\psi),$ and $f_s\in I(\tau,s,\psi^{-1})$. We define the zeta integrals, $\Psi(W,f_s).$ Note that \cite{Kap13a} defines integrals for any $n$ and $l$; however, we  only need the case of $n\leq l$ for the converse theorem and so we do not consider the case of $n>l$.

Now suppose that $n=l.$ Then we define
$$
\Psi(W,f_s)=\int_{ U_{\SO_{2l}}\backslash\SO_{2l}} W(g)f_s(w_{l,l}g, I_{l})dg,
$$
where 
$$
w_{l,l}=\left(\begin{matrix}
 \gamma I_{l} & &  \\
  & 1 & \\
  & & \gamma\inv I_l
\end{matrix}
\right)\in\SO_{2l+1}.$$
For any $g\in\SO_{2l},$ the integral satisfies \begin{equation}\label{zetan=l}
    \Psi(g\cdot W,g\cdot f_s)=\Psi(W,f_s).
\end{equation}

Next suppose that $n<l.$ Then we define

$$
\Psi(W,f_s)=\int_{ U_{\SO_{2n+1}}\backslash\SO_{2n+1}} \left(\int_{r\in R^{l,n}} W(r w^{l,n} g (w^{l,n})\inv)dr\right)   f_s(g, I_{l})dg,
$$
where
$$
w^{l,n}=\left(\begin{matrix}
 & I_n & & & \\
 I_{l-n-1} & & & & \\
 & & I_2 & & \\
 & & & & I_{l-n-1} \\
 & & & I_n &
\end{matrix}
\right)\in\SO_{2l},$$
and
$$
R^{l,n}=\left\{\left(\begin{matrix}
I_n & & & & \\
x & I_{l-n-1} & & & \\
& & I_2 & & \\
& & & I_{l-n-1} & \\
& & & x^\prime & I_n
\end{matrix}\right)\in\SO_{2l}\right\}.
$$
The integral satisfies the property \begin{equation}\label{zetan<l}
    \Psi((gn)\cdot W,g\cdot f_s)=\psi\inv(n)\Psi(W,f_s),
\end{equation} for any $g\in\SO_{2n+1}$ and $n\in N^{l-n}.$ Note that in the case $n<l$, our integral differs from \cite{Kap13a} slightly. The difference is a right translation of the Whittaker functional by $(w^{l,n})\inv.$ This is not a significant issue since right translation preserves $\mathcal{W}(\pi,\psi);$ however, the translation simplifies the calculations of the integrals later. Moreover, this does not affect our definition of the $\gamma$-factor below.

We introduce the standard intertwining operator $M(\tau,s,\psi^{-1})$. Note that we do not use the normalized version of the intertwining operator. We define $M(\tau,s,\psi^{-1}): I(\tau,s,\psi^{-1})\rightarrow I(\tau^*,1-s,\psi^{-1})$ via
$$
M(\tau,s,\psi^{-1})f_s(h,a)=\int_{V_n} f_s(w_n u h, d_n a^*)du,
$$
where $w_n=\left(\begin{matrix}
 & & I_n \\
 & (-1)^n & \\
 I_n & &
\end{matrix}\right)$, $d_n=\mathrm{diag}(-1,1,-1,\dots,(-1)^n)\in\GL_n$, and $V_n$ is the unipotent radical of the parabolic subgroup $Q_n=L_n V_n$ in $\SO_{2n+1}$ where $L_n\cong \GL_n$, and $a^*=J_n {}^t a^{-1} J_n.$

Outside of a discrete subset of $s$, the set of bilinear forms satisfying equations (\ref{zetan=l}) or (\ref{zetan<l}) is at most $1$ dimensional. This follows from the results of \cite{AGRS10} and \cite{MW12} (see \cite{GGP12a}). When $n=l$, this was already known in \cite{GPSR87}.

Therefore, we can define a proportionality factor, $\gamma(s, \pi\times\tau, \psi)$, called the $\gamma$-factor, such that
$$
\gamma(s,\pi\times\tau, \psi)\Psi(W,f_s)=\Psi(W,M(\tau,s,\psi^{-1})f_s).
$$
We refer to these integrals and $\gamma$-factors as the twists by $\GL_n.$

Note that our $\gamma$-factor differs from that of \cite{Kap13a} slightly. Specifically, it differs by a factor depending only on $\tau$ and $\psi$. If we let $\gamma'(s,\pi\times\tau, \psi)$ be the gamma factor of \cite{Kap13a}, then 
$$
\gamma(\tau, Sym^2,\psi,2s-1)\gamma(s,\pi\times\tau, \psi)=c(n,\tau,\psi,\gamma)\gamma'(s,\pi\times\tau, \psi),
$$
where for $n<l$ we have $c(n,\tau,\psi,\gamma)=\omega_\tau(\gamma)^2 |\gamma|^{2n(s-\frac{1}{2})}$ and $c(l,\tau,\psi,\gamma)=1.$ Here 
$\gamma(\tau, Sym^2,\psi,2s-1)$ is Shahidi's local coefficient, $ Sym^2$ is the symmetric square representation, and $\omega_\tau$ is the central character of $\tau.$ Thus, if the local converse theorem holds for $\gamma(s,\pi\times\tau, \psi)$, then it also  holds for $\gamma'(s,\pi\times\tau, \psi).$ There is also the normalization of $\gamma'(s,\pi\times\tau, \psi).$ It is defined by Kaplan in \cite[\S3.2]{Kap13a} by 
$$
\Gamma(s,\pi\times\tau, \psi)=\omega_\pi(-1)^n\omega_\tau(-1)^l\gamma'(s,\pi\times\tau, \psi).
$$ It follows that local converse theorems for $\gamma(s,\pi\times\tau, \psi)$, $\gamma'(s,\pi\times\tau, \psi),$ and $\Gamma(s,\pi\times\tau, \psi)$ are all equivalent.

\section{Partial Bessel functions}\label{Partial Bessel functions}
In this section, we introduce Howe vectors and partial Bessel functions. These played a crucial role in the proof of the local converse theorems considered in \cite{HL22b, Jo22, Zha18, Zha19}. Their properties again play an important role in the manipulations of the zeta integrals in this paper. Much of the setup for Howe vectors and partial Bessel functions is carried out in the split case in \cite{Bar95}. Unfortunately, the quasi-split case was not considered. Nevertheless, we generalize the relevant ideas to our case.
 
\subsection{Howe vectors}\label{Howe vectors}
Let $\omega$ be a character of the center, $Z$, of $\SO_{2l}$ and $C_c^\infty(\SO_{2l},\omega)$ be the space of smooth compactly supported functions $f$ satisfying $f(zg)=\omega(z)f(g)$ for any $z\in Z$ and $g\in\SO_{2l}.$ Recall that we fixed $\psi$ to be a nontrivial unramified additive character of $F$ and used it to define a character, also denoted by $\psi$, on $U_{\SO_{2l}}.$ Let $C^\infty(\SO_{2l},\omega,\psi)$ be the space of functions $W$ satisfying $W(zg)=\omega(z)W(g)$ and $W(ug)=\psi(u)W(g)$ for any $z\in Z,$ $u\in U_{\SO_{2l}}$, and $g\in\SO_{2l},$ and furthermore that there exists an open compact subgroup $K$ of $\SO_{2l}$ for which $W$ is invariant under right translation.

Let $m>0$ be an integer and define congruence subgroups 
\begin{align*}
    K_m^{\GL_{2l}}&=I_{2l}+\mathrm{Mat}_{2l\times 2l}(\mathfrak{p}^m), \\
    K_m&=K_m^{\GL_{2l}}\cap \SO_{2l}.
\end{align*}
We define a character $\tau_m$ on $K_m$ by setting 
$$
\tau_m(k)=
   \psi\left(\varpi^{-2m} \left(\sum_{i=1}^{l-2} k_{i,i+1}+\frac{1}{2} k_{l-1,l+1}  \right)\right),
$$
where $k=(k_{i,j})_{i,j=1}^{2l}\in K_m.$
Set
$$
e_m=\mathrm{diag}(\varpi^{-2m(l-1)},\dots,\varpi^{-2m},1,1,\varpi^{2m},\dots,\varpi^{2m(l-1)}).
$$
Let $H_m=e_m K_m e_m\inv$ and define a character $\psi_m$ on $H_m$ via $\psi_m(h)=\tau_m(e_m\inv h e_m)$ for $h\in H_m.$ Let $U_m=U_{\SO_{2l}}\cap H_m.$ Note that $U_{\SO_{2l}}=\cup_{m\geq 1} U_m$ and the character $\psi$ on $U_{\SO_{2l}}$ agrees with $\psi_m$ upon restriction to $U_m.$

We fix the split part, denoted $S_{\SO_{2l}}$, of the torus $T_{\SO_{2l}}$ to be the subset consisting of elements of the form $t=\mathrm{diag}(t_1,\dots,t_{l-1},I_2,t_{l-1}\inv,\dots,t_1\inv).$ Let $W(\SO_{2l})$ be the Weyl group of $\SO_{2l}$ and let $\Delta(\SO_{2l})$ be the set of simple roots with respect to $S_{\SO_{2l}}$. We can choose the simple roots to be given by $\alpha_i(t)=t_it_{i+1}\inv$ for $i=1,\dots,l-2$ and $\alpha_{l-1}(t)=t_{l-1}.$ In the general case of non-split groups, the root subgroups are not guaranteed to be isomorphic to $F$ and so we define elements of $U_{\SO_{2l}}$ that we need here.  Let $x\in F.$ For $i=1,\dots,l-2$ we define $\mathrm{\mathbf{x}}_{\alpha_i}(x)=(u_{i,j})_{i,j=1}^{2l}\in U_{\SO_{2l}}$ where $u_{j,j}=1,$ for any $j=1,\dots,2l,$ $ u_{i,i+1}=x, u_{2l-i,2l-(i-1)}=-x,$ and $u_{i,j}=0$ for any other pair $(i,j).$ We also define
\begin{align*}
    \mathrm{\mathbf{x}}_{\alpha_{l-1}}(x)=\left(\begin{matrix} I_{l-2} & & \\
    & \left(\begin{matrix}
        1 & 0 & x &\frac{\rho\inv x^2}{2} \\
        0 & 1 & 0 & 0 \\
        0 & 0 & 1 & \rho\inv x \\
        0 & 0 & 0 & 1
    \end{matrix}\right) & \\
    & & I_{l-2}\end{matrix}\right).
\end{align*}
It is checked directly that $\mathrm{\mathbf{x}}_{\alpha}(x)\in U_\alpha,$ where $U_\alpha$ denotes the root subgroup of $\alpha,$ for any $\alpha\in\Delta(\SO_{2l}).$

Now let $\alpha$ be any root of $\SO_{2l}$ and $U_\alpha$ be the root space of $\alpha.$ Set $U_{\alpha,m}=H_m\cap U_\alpha.$ For any simple root $\alpha$ and $x\in F,$ we verify directly that $\mathrm{\mathbf{x}}_{\alpha}(x)\in U_{\alpha,m}$ if and only if $x\in\mathfrak{p}^{-m}.$ This is an analogous result to \cite[Lemma 3.1]{Bar95}.

Any matrix in $K_m^{\GL_{2l}}$ can be decomposed uniquely as the product of a unipotent upper triangular matrix and a lower triangular matrix. This property is also inherited by $K_m$ and $H_m.$ Let $\overline{B}_{\SO_{2l}}$ be the opposite parabolic subgroup to $B_{\SO_{2l}}.$ We record the above fact in the following lemma.

\begin{lemma}\label{lemma UL decomposition}
We have
$$
    K_m=(U_{\SO_{2l}}\cap K_m)(\overline{B}_{\SO_{2l}}\cap K_m), \ \  H_m=(U_{\SO_{2l}}\cap H_m)(\overline{B}_{\SO_{2l}}\cap H_m). 
$$
\end{lemma}

Let $W\in C^\infty(\SO_{2l},\omega,\psi)$ with $W(I_{2l})=1.$ We define 
$$
W_m(g)=\frac{1}{\mathrm{Vol}(U_m)}\int_{U_m}\psi_m\inv(u) W(gu)du.
$$
Let $C=C(W)$ be the maximal positive integer such that $W$ is invariant under right translation by $K_C.$ If $m\geq C,$ then $W_m$ is called a Howe vector. The following lemma is the analogue of {\cite[Lemma 3.2]{Bar95}}.

\begin{lemma}\label{partialBesselprop}
We have
\begin{enumerate}
    \item $W_m(I_{2l})=1,$
    \item if $m\geq C,$ then  $W_m(gh)=\psi_m(h)W_m(g)$ for any $h\in H_m$ and $g\in\SO_{2l}.$
\end{enumerate}
\end{lemma}

\begin{proof}
    Part (1) follows directly since $W(u)=\psi(u)W(I_{2l})=\psi(u)$ and $\psi(u)=\psi_m(u)$ for any $u\in U_m.$
    
    For Part (2), set
    \begin{align*}
W_m'(g)&=\frac{1}{\mathrm{Vol}(H_m)}\int_{H_m}\psi_m\inv(h) W(gh)dh.
    \end{align*}
    Since $m\geq C$ and $\overline{B}_{\SO_{2l}}\cap H_m\subseteq K_m$, by Lemma \ref{lemma UL decomposition},
        \begin{align*}
W_m'(g)&=\frac{1}{\mathrm{Vol}(U_m)\mathrm{Vol}(\overline{B}_{\SO_{2l}}\cap H_m)}\int_{U_m}\int_{\overline{B}_{\SO_{2l}}\cap H_m}\psi_m\inv(u) W(gub)dudb \\
&= \frac{1}{\mathrm{Vol}(U_m)}\int_{U_m}\psi_m\inv(u) W(gu)du \\
&=W_m(g).
    \end{align*}
From the definition, we have $W_m'(gh)=\psi_m(h)W_m'(g)$ and hence it also follows that $W_m(gh)=\psi_m(h)W_m(g).$
\end{proof}

Hence, for $m\geq C,$ we have
\begin{equation}\label{Besseleqn}
    W_m(ugh)=\psi(u)\psi_m(h)W_m(g),
\end{equation}
for any $u\in U_{\SO_{2l}}$, $h\in H_m$, and $g\in\SO_{2l}.$ For $m\geq C$, $W_m$ is called a {\it partial Bessel function}.

The next lemma shows that the Whittaker model of an irreducible generic
representation $\pi$ contains partial Bessel functions.

\begin{lemma}\label{lemma partial Bessel in Whittaker model}
    Let $\pi$ be an irreducible generic representation of $\SO_{2l}$ and $W\in\mathcal{W}(\pi,\psi)$ with $W(I_{2l})=1.$ Then we have $W_m\in\mathcal{W}(\pi,\psi).$
\end{lemma}

\begin{proof}
    Let $V_\pi$ be the vector space on which the representation $\pi$ acts. Recall that we fixed $\Gamma\in\mathrm{Hom}_{U_{\SO_{2l}}}(\pi,\psi)$ to be a nonzero Whittaker functional. By definition, there exists $v\in V_\pi$ such that $W(g)=W_v(g)=\Gamma(\pi(g)v)$ for any $g\in\SO_{2l}.$ Since $W_v(I_{2l})=1,$ we have $\Gamma(v)=1.$
Set  
    \begin{align*}
v_m&=\frac{1}{\mathrm{Vol}(U_m)}\int_{U_m}\psi_m\inv(u) \pi(u)vdu.
    \end{align*}
    Then $W_{v_m}\in\mathcal{W}(\pi,\psi)$ and
    \begin{align*}
        W_{v_m}(g)=\Gamma(\pi(g)v_m)&=\frac{1}{\mathrm{Vol}(U_m)}\int_{U_m}\psi_m\inv(u) \Gamma(\pi(gu)v)dh \\
        &= \frac{1}{\mathrm{Vol}(U_m)}\int_{U_m}\psi_m\inv(u) W_v(gu)du \\
        &= W_m(g).
    \end{align*}
Thus $W_m\in\mathcal{W}(\pi,\psi).$
\end{proof}

Let $f\in C_c^\infty(\SO_{2l},\omega).$ We define $W^f\in C^\infty(\SO_{2l},\omega,\psi)$ by
$$
W^f(g)=\int_{U_{\SO_{2l}}} \psi\inv(u)f(ug)du,
$$
for $g\in\SO_{2l}.$ Since $f$ is locally constant and compactly supported, there exists a positive integer $C=C(f)$ such that $W^f$ is invariant under right translation by $K_C.$ We choose $f\in C_c^\infty(\SO_{2l},\omega)$ such that $W^f(I_{2l})=1.$ Then, for $m>C$, we consider the corresponding partial Bessel function defined by
$$
B_m(g;f):=(W^f)_m(g)=\frac{1}{\mathrm{Vol}(U_m)}\int_{ U_{\SO_{2l}}\times U_m}\psi\inv(u)\psi_m\inv(u') f(ugu')dudu'.
$$

\subsection{Conjugate of partial Bessel Function}
Let $\psi_c$ be the character on $U_{\SO_{2l}}$ defined by $\psi_c(u)=\psi(cuc).$ That is, 
$$
\psi_c(u)= \psi\left(\sum_{i=1}^{l-2} u_{i,i+1}-\frac{1}{2}u_{l-1,l+1}  \right),
$$
for any $u=(u_{i,j})_{i,j=1}^{2l}\in U_{\SO_{2l}}.$ Recall that we fixed a nonzero $\Gamma\in\mathrm{Hom}_{U_{\SO_{2l}}}(\pi,\psi).$ Then $\Gamma\in\mathrm{Hom}_{U_{\SO_{2l}}}(\pi^c,\psi_c)$ is also nonzero. Note that $\pi^c$ is also $\psi$-generic. We relate the partial Bessel functions of $\pi$ and $\pi^c$ below.

 Fix $\tilde{t}:=\mathrm{diag}(I_{l-1},-1,-1,I_{l-1})\in T_{\SO_{2l}}.$ It is checked directly that $\psi_c(\tilde{t}\inv u\tilde{t})=\psi(u).$  We also have $c K_m c = K_m,$ $c U_m c = U_m,$ and $\psi_m(c\tilde{t}\inv u \tilde{t}c)=\psi_m(u).$ Given $W\in C^\infty(\SO_{2l},\omega,\psi)$, define $W^c\in C^\infty(\SO_{2l},\omega,\psi)$ by
$W^c(g)=W(c\tilde{t}\inv g \tilde{t}c).$ If $W(I_{2l})=1,$ we also define $W_m^c=(W^c)_m.$ For $f\in C_c^\infty(\SO_{2l},\omega)$, we set $W^{f,c}=(W^f)^c$ and if $W^f(I_{2l})=1,$ then we set $B_m^c(g;f)= (W^{f,c})_m.$

\begin{lemma}\label{lemma conj bessel}
Suppose $W\in C^\infty(\SO_{2l},\omega,\psi)$ and $W(I_{2l})=1.$ Then,
\begin{enumerate}
    \item $W^c(I_{2l})=1,$
    \item $W^c_m=(W_m)^c,$
    \item $C(W)=C(W^c),$
    \item for $m\geq C,$ we have
$$
W_m^c(ugh)=\psi(u)\psi_m(h)W_m^c(g),
$$
for any $u\in U_{\SO_{2l}}$, $h\in H_m$, and $g\in\SO_{2l}.$
\end{enumerate}
\end{lemma}

\begin{proof}
$(1)$ is immediate from the definition. By definition and since $c\tilde{t}\inv U_m\tilde{t} c = U_m$, we obtain $(2).$  $(3)$ follows from the fact $c K_m c = K_m.$ Finally, $(4)$ is verified directly from the previous results and definitions. \end{proof}

Thus, we can see that given a partial Bessel function, $W_m$ or $B_m,$ we may construct another partial Bessel function via the map $W_m\mapsto W_m^c$ or $B_m\mapsto B_m^c.$ In fact, this map sends partial Bessel functions in the Whittaker model of $\pi$ to partial Bessel functions in the Whittaker model of $\pi^c.$ Consequently, by the uniqueness of Whittaker models, this map is the identity map on the Whittaker model of $\pi$ if and only if $\pi\cong \pi^c.$

\subsection{Bessel Support}\label{Besselsupp}

Recall $W(\SO_{2l})$ is the Weyl group and $\Delta(\SO_{2l})$ is the set of simple roots of $\SO_{2l}$. We say that an element $w\in W(\SO_{2l})$ supports partial Bessel functions if for any $\alpha\in\Delta(\SO_{2l})$, $w\alpha$ is either negative or simple. Set B($\SO_{2l}$) to be the set of Weyl elements which support partial Bessel functions. B($\SO_{2l}$) is called the Bessel support. Note that the partial Bessel functions which we are interested in vanish on the Bruhat cells corresponding to Weyl elements outside of the Bessel support (\cite[Proposition 8.6]{CPSSh05}).

The standard maximal parabolic subgroups of  $\SO_{2l}$ are $P_n=M_n N_n$ for $n\leq l-1$ where $M_n$ is the standard Levi subgroup that is isomorphic to $\GL_n\times\SO_{2(l-n)}$. Let $w_{long}$ denote the long Weyl element of $\SO_{2l}$ and $w_{M_n}$ be the long Weyl element in $M_n.$ Set $\tilde{w}_n:=w_{long}\inv w_{M_n}.$ The analogues of these elements played an important role in partitioning the Bessel support for other groups both locally and over finite fields (\cite{Haz23a, HL22a, HL22b, Jo22, LZ22a, Zha18, Zha19}) and they continue to do so in our case. To be explicit, we have
$$
w_{long}=\left(\begin{matrix}
 & & J_{l-1} \\
 & \left(\begin{matrix}
 -1 & 0 \\
 0 & 1 
 \end{matrix}\right)^{l-1} & \\
 J_{l-1} & &
\end{matrix}\right), 
w_{M_n}=\left(\begin{matrix}
J_n & &  \\
 & w_{long,2(l-n)} & \\
  & & J_n
\end{matrix}\right),$$
where $w_{long,2(l-n)}$ is the long Weyl element of $\SO_{2(l-n)}.$
Hence,
$$
\tilde{w}_n=\left(\begin{matrix}
& & & & I_n \\
 & I_{l-n-1} &  & & \\
 & & \left(\begin{matrix}
 -1 & 0 \\
 0 & 1 
 \end{matrix}\right)^{n} & & & \\
 & & & I_{l-n-1} &  \\
 I_n & & & &
\end{matrix}\right).
$$
Recall that the definition of the intertwining operator $M(\tau,\psi\inv)$ involved the Weyl element $w_n$ of $\SO_{2n+1}.$ The image of $w_n$ under the embedding of $\SO_{2n+1}$ into $\SO_{2l}$ and conjugated by $w^{l,n}$ is $\tilde{w}_n$ (this conjugation is included in the definition of our zeta integrals). That is, $w^{l,n} w_n (w^{l,n})^{-1}=\tilde{w}_n.$

Let $t$ be in the split part of the torus and write
$
t=\mathrm{diag}(t_1,\dots,t_{l-1},I_2,t_{l-1}\inv,\dots,t_1\inv).
$ 
Then 
$
\tilde{w}_n\inv t \tilde{w}_n=\mathrm{diag}(t_n\inv,\dots,t_1\inv,t_{n+1},\dots,t_{l-1},I_2,t_{l-1}\inv,\dots,t_{n+1}\inv,t_1,\dots,t_n).
$
Thus, we have
$$\tilde{w}_n\alpha_i = \alpha_{n-i} \, \, \mathrm{for} \, \, 1\leq i \leq n-1,$$
$$\tilde{w}_n\alpha_n (t) = t_1\inv t_{n+1}\inv,$$ 
$$\tilde{w}_n\alpha_i=\alpha_i \, \, \mathrm{for} \, \, n+1\leq i \leq l-1.$$ 

Let $\theta_w=\{\alpha\in\Delta(\SO_{2l}) \, | \, w \alpha\in\Phi^+(\SO_{2l})\}.$ The map $w\mapsto \theta_w$ defines a bijection from $\mathrm{B}(\SO_{2l})$ to the power set of $\Delta(\SO_{2l})$, which we denote by $\mathcal{P}(\Delta(\SO_{2l}))$. For $n\leq l-1,$ we have $\theta_{\tilde{w}_n}=\Delta(\SO_{2l})\setminus \{\alpha_n\}$.

For $n\leq l-1,$ let $\mathrm{B}_n(\SO_{2l})$ be the set of $w\in \mathrm{B}(\SO_{2l})$ such that there exists $w'\in W(\GL_n)$ such that $w=t_n(w')\tilde{w}_n.$ We set $\mathrm{B}_0(\SO_{2l})=\{I_{2l}\}.$ Also, we define 
$
P_n=\{ \theta\subseteq\Delta(\SO_{2l}) \, | \, w_\theta\in \mathrm{B}_n(\SO_{2l})\}.
$
The following proposition is implicit in \cite{LZ22a}.
\begin{prop}[{\cite[Proposition 4.5]{Haz23a}}]\label{Besselprtnon}
$$P_n=\{ \theta\subseteq\Delta(\SO_{2l}) \, | \, \{\alpha_{n+1},\dots,\alpha_{l-1}\} \subseteq \theta \subseteq \Delta(\SO_{2l})\setminus \{\alpha_n\}\}.$$
\end{prop}

It follows immediately that the sets $\mathrm{B}_n(\SO_{2l})$ for $n\leq l-1$ partition the Bessel support.
\begin{cor}[{\cite[Corollary 4.6]{Haz23a}}]\label{Besselpartnonsplit} The sets $\mathrm{B}_n(\SO_{2l})$ for $n\leq l-1$ form a partition of the Bessel support $\mathrm{B}(\SO_{2l}).$
\end{cor}

\subsection{Bruhat Order}\label{Bruhat order}
Much of the discussion in this section follows \cite{HL22b, Zha18, Zha19}. For $w\in W(\SO_{2l}),$ let $C(w)=B_{\SO_{2l}}wB_{\SO_{2l}}$. We define the Bruhat order on $W(\SO_{2l})$ by $w \leq w'$ if and only if $C(w)\subseteq \ol{C(w')}.$ Then
$$
\ol{C(w')}=\bigsqcup_{
w\leq w'} C(w).
$$
We also let $$\Omega_w = \bigsqcup_{ w'\geq w} C(w').$$ Consequently, $C(w)$ is closed in $\Omega_w.$ 

The following Proposition is \cite[Proposition 4.3]{Zha18}. There it was proved for symplectic groups, but the proof is general enough to apply directly to our case.

\begin{prop}[{\cite[Proposition 4.3]{Zha18}}]\label{Omega_w}
\begin{enumerate}
\item[]
    \item If $w,w'\in W(\SO_{2l})$ with $w'>w$ then $\Omega_{w'}$ is an open subset of $\Omega_w.$ In particular, for any $w\in W(\SO_{2l})$, $\Omega_w$ is open in $\Omega_{I_{2l}}=\SO_{2l}.$
    \item Let $P$ be a standard parabolic subgroup of $\SO_{2l}$ and $w\in W(\SO_{2l}).$ Then $PwP\cap\Omega_w$ is closed in $\Omega_w.$
\end{enumerate}
\end{prop}

Suppose that $w\in \mathrm{B}(\SO_{2l}).$ Then, $w_l w$ is the longest Weyl element of some standard Levi subgroup which we denote by $M_w.$ 

\begin{lemma}[{\cite[Proposition 2.1]{CPSSh05}}]\label{CPSSh}
Let $w,w'\in \mathrm{B}(\SO_{2l}).$ Then $w'\leq w$ if and only if $M_{w}\subseteq M_{w'}$ if and only if $\theta_{w}\subseteq\theta_{w'}.$
\end{lemma}

Suppose $w,w'\in \mathrm{B}(\SO_{2l})$ with  $w'<w.$ Following \cite{Jac16},  define the Bessel distance between $w$ and $w'$ as
$$
d_B(w,w')=\mathrm{max}\{m \, | \, \exists \, w_i\in \mathrm{B}(\SO_{2l}) \, \mathrm{with} \, w=w_m > w_{m-1} > \cdots > w_0=w' \}.
$$ 
From the previous lemma, if $d_B(w,I_{2l})=1,$ then there exists $\alpha\in\Delta(\SO_{2l})$ such that $\theta_w=\Delta(\SO_{2l})\backslash\{\alpha\}.$ Since there are $l-1$ simple roots,  $w$ must be one of $\tilde{w}_n$ for $n\leq l-1$.
For $w_1,w_2\in W(\SO_{2l})$ with $w_1\leq w_2$, define the closed Bruhat interval as
$$
[w_1,w_2]=\{w\in W(\SO_{2l}) \, | \, w_1\leq w \leq w_2 \}. 
$$
For $n\leq l-1,$ let $P_{n, 1^{l-n}}$ be the standard parabolic subgroup of $\SO_{2l}$ with Levi subgroup $L_{n,1^{l-n}}\cong \GL_n\times\GL_1^{l-n-1}\times\SO_2.$

\begin{lemma}\label{P_{k, 1^{l-k}}}
For $n\leq l-1,$ the set 
$$
\{w\in W(\SO_{2l}) \, | \, C(w) \subseteq P_{n, 1^{l-n}} \tilde{w}_n P_{n, 1^{l-n}}\}
$$
is equal to the Bruhat interval $[\tilde{w}_n, w_l^{L_{n, 1^{l-n}}}\tilde{w}_n]$ where $w_l^{L_{n, 1^{l-n}}}$ is the long Weyl element in $L_{n, 1^{l-n}}.$
\end{lemma}

\begin{proof}
This proof is similar to \cite[Lemma 4.5]{Zha18}; however, we provide it here for completeness.
By \cite[Proposition 2]{BKPST18}, the set $$D=\{w\in W(\SO_{2l}) \, | \, C(w) \subseteq P_{n, 1^{l-n}} \tilde{w}_n P_{n, 1^{l-n}}\}
$$ is a Bruhat interval $[w_{\mathrm{min}},w_{\mathrm{max}}].$ The simple roots of $L_{n, 1^{l-n}}$ are precisely $\Delta_{L_{n, 1^{l-n}}}=\{\alpha_1,\dots,\alpha_{n-1}\}.$ Let $W(\Delta_{L_{n, 1^{l-n}}})$ be the Weyl group for $L_{n, 1^{l-n}}.$ Then, by \cite[Corollary 5.20]{BT65}, $D$ is the double coset $W(\Delta_{L_{n, 1^{l-n}}})\tilde{w}_n W(\Delta_{L_{n, 1^{l-n}}}).$ We have $\tilde{w}_n=\tilde{w}_n\inv$ and $\tilde{w}_n(\Delta_{L_{n, 1^{l-n}}})>0.$ From \cite[Proposition 2(2)]{BKPST18}, we have $w_{\mathrm{min}}=\tilde{w}_n.$ 

By \cite[Corollary 3]{BKPST18}, every $w\in W(\Delta_{L_{n, 1^{l-n}}})\tilde{w}_n W(\Delta_{L_{n, 1^{l-n}}})$ can be written uniquely as $w=\tilde{w}_n w'$ where $w'\in W(\Delta_{L_{n, 1^{l-n}}}).$ Furthermore, $l(w)=l(\tilde{w}_n)+l(w')$ where $l$ denotes the length of the Weyl element. Hence, the maximal such $w$ is $w_{\mathrm{max}}=\tilde{w}_n w_l^{L_{n, 1^{l-n}}}= w_l^{L_{n, 1^{l-n}}} \tilde{w}_n.$
\end{proof}

The next theorem gives the analogues of \cite[Lemmas 5.13 and 5.14]{CST17}. Similar to the cases considered in \cite{HL22b, Zha18, Zha19}, the proof generalizes to our setting.

\begin{thm}[Cogdell-Shahidi-Tsai]\label{CST}
\begin{enumerate}
\item[]
    \item Let $w\in \mathrm{B}(\SO_{2l})$ and $f\in C_c^\infty(\Omega_w,\omega).$ Suppose $B_m(tw,f)=0$ for any $t\in T_{\SO_{2l}}$. Then, there exists $f'\in C_c^\infty(\Omega_w\backslash C(w),\omega)$, depending only on $f$, such that for $m$ large enough and also only depending on $f$, $B_m(g,f)=B_m(g,f')$ for any $g\in\SO_{2l}.$
    \item Let $w\in \mathrm{B}(\SO_{2l})$ and $\Omega_{w,0},\Omega_{w,1}\subseteq\Omega_{w}$ be two open subsets which are $U_{\SO_{2l}}\times U_{\SO_{2l}}$ and $T_{\SO_{2l}}$-invariant such that $\Omega_{w,0}\subseteq\Omega_{w,1}$ and $\Omega_{w,1}\backslash\Omega_{w,0}$ is a union of Bruhat cells $C(w')$ where $w'\notin \mathrm{B}(\SO_{2l}).$ Then, for any $f_1\in C_c^\infty(\Omega_{w,1},\omega),$ there exists $f'\in C_c^\infty(\Omega_{w,0},\omega)$ such that for $m$ sufficiently large and depending only on $f$, $B_m(g,f_1)=B_m(g,f')$ for any $g\in\SO_{2l}.$
\end{enumerate}
\end{thm}
The following corollary is an analogue of \cite[Proposition 5.3]{CST17}.

\begin{cor}\label{C(0)}
Let $f,f'\in C_c^\infty(\SO_{2l},\omega)$ such that $W^f(I_{2l})=W^{f'}(I_{2l})=1.$ Then, for any $1\leq n \leq l-1$ there exists $f_{\tilde{w}_n}\in C_c^\infty(\Omega_{\tilde{w}_n},\omega)$, depending both on $f$ and $f'$,
such that for $m$ large enough 
$$
B_m(g,f)-B_m(g,f')=\sum_{n=1}^{l-1} B_m(g,f_{\tilde{w}_n}).
$$
\end{cor}

\begin{proof}
The proof is an adaptation of \cite[Corollary 4.7]{Zha18} which we omit. 
\end{proof}

\subsection{Outline of the proof of Theorem \ref{converse thm intro}}\label{outline of the proof}

In this section, we outline the strategy of the proof of Theorem \ref{converse thm intro}. Let $\mathcal{C}(0)$ denote the condition that $\pi$ and $\pi'$ have the same central character $\omega$. We always assume $\mathcal{C}(0).$ For $n\geq 1,$ define the condition $\mathcal{C}(n)$ by $\gamma(s,\pi\times\tau,\psi)=\gamma(s,\pi'\times\tau,\psi)$, as functions of the complex variable $s$, for all irreducible generic representations $\tau$ of $\GL_k$ with $1\leq k \leq n$ along with $\mathcal{C}(0).$ It is clear that $\mathcal{C}(n)$ implies $\mathcal{C}(n-1).$ As $\pi$ is supercuspidal, its matrix coefficients, denoted by $\mathcal{M}(\pi)$, have compact support. That is, $\mathcal{M}(\pi)\subseteq C_c^\infty(\SO_{2l},\omega).$ Since $\pi$ is generic, the map $\mathcal{M}(\pi)\rightarrow C^\infty(\SO_{2l},\omega,\psi)$ given by $f\mapsto W^f$ is nonzero.

Let $f\in\mathcal{M}(\pi)$ and $f'\in\mathcal{M}(\pi')$ such that $W^f(I_{2l})=W^{f'}(I_{2l})=1.$ By assumption, $\pi$ and $\pi'$  satisfy $\mathcal{C}(0).$ By Corollary \ref{C(0)}, there exists $f_{\tilde{w}_n}\in C_c^\infty(\Omega_{\tilde{w}_n},\omega)$ for $1\leq n \leq l-1$  such that for $m$ large enough 
\begin{equation}\label{Bessel Difference}
B_m(g,f)-B_m(g,f')=\sum_{n=1}^{l-1} B_m(g,f_{\tilde{w}_n}).
\end{equation}

Similar to \cite{HL22b, Zha18, Zha19}, we show that the condition $\mathcal{C}(k)$ implies
$$
B_m(g,f)-B_m(g,f')=\sum_{n=k+1}^{l-1} B_m(g,f_{\tilde{w}_n}).
$$
for $k\leq l-1$ (Theorem \ref{l-2 theorem}). Thus $\mathcal{C}(l-2)$ gives
$$
B_m(g,f)-B_m(g,f')= B_m(g,f_{\tilde{w}_{l-1}}).
$$
However, unlike the symplectic and unitary cases,  $\mathcal{C}(l-1)$ does not imply that we can remove $B_m(g,f_{\tilde{w}_{l-1}}).$ Instead, it only shows that $B_m(\cdot,f_{\tilde{w}_{l-1}})=0$ on elements which are fixed by the outer conjugation by $c$ (Theorem \ref{l-1 thm}). For those elements which are not fixed by $c,$ we use the condition $\mathcal{C}(l).$ It determines that $B_m(\cdot,f_{\tilde{w}_{l-1}})+B_m^c(\cdot,f_{\tilde{w}_{l-1}})=0$ on all the elements which are not fixed by the outer conjugation (Theorem \ref{l thm}). As a result, we must consider the partial Bessel functions for $\pi^c$ and $\pi'^c.$ Consequently, we deduce Theorem \ref{converse thm intro} from the uniqueness of Whittaker models.

\section{Sections of induced representations}\label{sections}
Let $n\leq l$ and $(\tau, V_\tau)$ be an irreducible $\psi\inv$-generic representation of $\GL_n$. In this section, we construct a section in $I(\tau,s,\psi\inv)$ which is used in calculating the zeta integrals. Let $\ol{V}_n$ be the unipotent radical of the opposite parabolic to $Q_n$. That is,
$$
\ol{V}_n=\left\{\left(\begin{matrix}
I_n & & \\
* & 1 & \\
* & * & I_n
\end{matrix}\right)\in\SO_{2n+1}\right\}.
$$
Let $i$ be a positive integer and recall the definition of $H_i\subseteq\SO_{2l}$ in  \S\ref{Howe vectors}. For $n<l,$ set $\ol{V}_{n,i}=w^{l,n}\ol{V}_n(w^{l,n})\inv\cap H_i$ where we realize $\ol{V}_n$ through the embedding of $\SO_{2n+1}$ into $\SO_{2l}.$ 

Recall that for a positive integer $m$, we set $K_m^{\GL_n}$ to be the congruence subgroups of $\GL_n.$ That is, $K_m^{\GL_n}=I_{n}+\mathrm{Mat}_{n\times n}(\mathfrak{p}^m).$ Following \cite[\S 3.2]{Bar95}, we define an open compact subgroup of $\SO_{2l+1}$ by $$H_m^{\SO_{2l+1}}=e'_m(K_m^{\GL_{2l+1}}\cap\SO_{2l+1})(e'_m)\inv$$ where $e'_m=\mathrm{diag}(\varpi^{-2ml},\dots,\varpi^{2ml})\in\SO_{2l+1}.$ If $n=l,$ we set $\ol{V}_{l,i}=\ol{V}_l\cap H_i^{\SO_{2l+1}}.$

Let $D$ be an open, compact subgroup of $V_n$. For $x\in D$, we let 
$$
S(x,i)=\{\ol{y}\in\ol{V}_n \, | \, \ol{y}x\in Q_n \ol{V}_{n,i}\}.
$$ 
We collect several lemmas from \cite{HL22b} below.

\begin{lemma}[{\cite[Lemma 6.1]{HL22b}}]\label{5.1}
For any positive integer $m$, there exists an integer $i_1=i_1(D,m)$ such that for any $i\geq i_1, x\in D,$ and $\ol{y}\in S(x,i)$ we have 
$$
\ol{y}x=vl_n(a)\ol{y}_0
$$ where $v\in V_n, a\in K_m^{\GL_n}$, and $\ol{y}_0\in\ol{V}_{n,i}.$
\end{lemma}

\begin{lemma}[{\cite[Lemma 6.2]{HL22b}}]
There exists an integer $i_2=i_2(D)$ such that $S(x,i)=\ol{V}_{n,i}$ for any $x\in D$ and $i\geq i_2.$
\end{lemma}

For $v\in V_\tau$ define a function $f_s^{i,v}:\SO_{2n+1}\rightarrow V_\tau$ by
$$
f_s^{i,v}(g)=\left\{\begin{array}{cc}
\frac{|\mathrm{det}a|^{s-\frac{1}{2}}\tau(a)v}{W_v(I_n)} & \, \mathrm{if} \, g=u'l_n(a)\ol{u} \, \mathrm{where} \, u'\in V_n, a\in\GL_n, \ol{u}\in \ol{V}_{n,i}, \\
0 & \mathrm{otherwise.}
\end{array}\right.
$$

\begin{lemma}[{\cite[Lemma 6.3]{HL22b}}]\label{5.3}
For any $v\in V_\tau,$ there exists an integer $i_0(v)$ such that for any $i\geq i_0(v)$, $f_s^{i,v}\in I(\tau,s).$
\end{lemma}

Recall that we fixed a nonzero Whittaker functional $\Lambda_\tau \in \mathrm{Hom}_{U_{\GL_{n}}}(\tau,\psi^{-1}).$ We define a function $\xi_s^{i,v}:\SO_{2n+1}\times\GL_n\rightarrow\mathbb{C}$ by $\xi_s^{i,v}(g,a)=\Lambda_\tau(\tau(a)f^{i,v}_s(g)).$ Then, $\xi_s^{i,v}\in I(\tau,s,\psi\inv)$ and
\begin{equation}\label{xi nonintertwined}
\xi_s^{i,v}(g, I_n)=\left\{\begin{array}{cc}
\frac{|\mathrm{det}a|^{s-\frac{1}{2}}W_v(a)}{W_v(I_n)} & \, \mathrm{if} \, g=u'l_n(a)\ol{u} \,\, \mathrm{where} \, u'\in V_n, a\in\GL_n, \ol{u}\in \ol{V}_{n,i}, \\
0 & \mathrm{otherwise.}
\end{array}\right.
\end{equation}
Recall that $W_v(a)=\Lambda_\tau(\tau(a)v)$ is the Whittaker function of $\tau$ attached to $v.$

Next we consider the image of this function under the intertwining operator. Let $\tilde{\xi}_{1-s}^{i,v}=M(\tau,s,\psi\inv)\xi_s^{i,v}\in I(\tau^*,1-s,\psi\inv).$ 

\begin{lemma}[{\cite[Lemma 6.4]{HL22b}}]\label{5.4}
Let $D$ be an open compact subset of $V_n.$ Then there exists a positive integer $i(D,v)\geq i_0(v)$ such that for any $i\geq i(D,v)$, $\tilde{\xi}_{1-s}^{i,v}(w_n x, I_n)=\mathrm{Vol}(\ol{V}_{n,i}).$
\end{lemma}

The proof of the above lemma also describes the image of $\tilde{\xi}_{1-s}^{i,v}.$ We have
\begin{equation}\label{xi intertwined} \tilde{\xi}_{1-s}^{i,v}(u_1 l_n(a) w_n u_2, I_n)=\left\{\begin{array}{cc}
\mathrm{Vol}(\ol{V}_{n,i})|\mathrm{det}a|^{\frac{1}{2}-s}W_v^*(a) & \, \mathrm{if} \, \, u_1\in V_n, a\in\GL_n, u_2\in D, \\
0 & \mathrm{otherwise.}
\end{array}\right.
\end{equation}

\section{Proof of the local converse theorem}\label{proof of the local converse theorem}

Throughout this section, we continue to let $\pi$ and $\pi'$ be irreducible $\psi$-generic supercuspidal representations of $\SO_{2l}$ with same central character $\omega.$ We also fix matrix coefficients $f\in\mathcal{M}(\pi)$ and $f'\in\mathcal{M}(\pi')$ such that $W^f(I_{2l})=W^{f'}(I_{2l})=1$.

The goal of this section is to prove Theorem \ref{converse thm intro}. For twists up to $\GL_{l-2},$ the arguments are similar to those of \cite{HL22b,Zha18,Zha19}. We obtain that $\mathcal{C}(l-2)$ implies that
$$
B_m(g,f)-B_m(g,f')= B_m(g,f_{\tilde{w}_{l-1}}).$$
To compute the right hand side, we generalize methods of \cite{HL22b}. The twists up to $\GL_{l-1}$ determine the right hand side on certain elements that are fixed by the outer conjugation $c$. This is verified directly through explicit computation of the embedding. To obtain elements which are not fixed by conjugation, we use the twists by $\GL_l$; however, the result instead implies that the right hand side plus its outer conjugation by $c$ vanishes. That is,
$$
B_m(g,f_{\tilde{w}_{l-1}}) +B_m^c(g,f_{\tilde{w}_{l-1}})=0.
$$

\subsection{\texorpdfstring{Twists up to $\GL_{l-2}$}{}}

We consider the twists up to  $\GL_{l-2}.$ Recall from Equation \eqref{Bessel Difference},
$$
B_m(g,f)-B_m(g,f')=\sum_{n=1}^{l-1} B_m(g,f_{\tilde{w}_n}).
$$
In this subsection, we prove the following theorem.

\begin{thm}\label{l-2 theorem}
Suppose that $k\leq l-2.$ If $\pi$ and $\pi'$ satisfy the condition $\mathcal{C}(k)$, then
$$
B_m(g,f)-B_m(g,f')=\sum_{n=k+1}^{l-1}  B_m(g,f_{\tilde{w}_n}).
$$ for any $g\in\SO_{2l}$ and $m$ large enough.
\end{thm}

The proof is by induction. The case $k=0$ is already done by Corollary \ref{C(0)}. Next, we assume inductive hypothesis. That is, we assume that the condition $\mathcal{C}(k-1)$ implies
$$
B_m(g,f)-B_m(g,f')=\sum_{n=k}^{l-1} B_m(g,f_{\tilde{w}_n}) 
$$ for any $g\in\SO_{2l}$ and $m$ large enough. We will show that the condition $\mathcal{C}(k)$ implies
$$
B_m(g,f)-B_m(g,f')=\sum_{n=k+1}^{l-1}  B_m(g,f_{\tilde{w}_n}).
$$

We begin with some preliminaries.
For $k\leq l-1,$ let $P_k$ be the standard maximal parabolic subgroup of $\SO_{2l}$ with Levi subgroup $M_k$ isomorphic to $\GL_k\times\SO_{2(l-k)}.$ Let $N_k$ be its unipotent radical. Furthermore, we let
$$
S^+_k=\left\{\left(\begin{matrix}
u_1 & &  \\
    & u_2 & \\
    &   & u_1^*
\end{matrix}\right) \, \left| \, u_1\in U_{\GL_k}, u_2\in U_{\SO_{2(l-k)}}\right\}\right. .
$$
Note that $U_{\SO_{2l}}=N_k S_k^+=S_k^+N_k .$ The product map gives an isomorphism
$$
P_k \times \{\tilde{w}_k\}\times N_k \rightarrow P_k \tilde{w}_k P_k.
$$
Alternatively, one can check that conjugation by $\tilde{w}_k$ takes $P_k$ to its opposite parabolic.

\begin{lemma}\label{6.2}
For $u_0\in N_k\backslash(N_k\cap U_m)$, $u^+\in S_k^+\cap U_m$, we have that $(u^+)\inv u_0 u^+\in N_k\backslash(N_k\cap U_m).$
\end{lemma}

\begin{proof}
The proof is similar to \cite[Lemma 6.2]{Zha18}.
\end{proof}

Recall that, for $n\leq l-1,$ we set $P_{n, 1^{l-n}}$ to be the standard parabolic subgroup with Levi subgroup $L_{n,1^{l-n}}\cong \GL_n\times\GL_1^{l-n-1}\times\SO_2.$ Let $m$ be a positive integer and recall the definition of $U_m$ from $\S$\ref{Howe vectors}. For $a\in\GL_n$, we let
$$
t_n(a)=\mathrm{diag}(a,I_{2(l-n)},a^*)\in\SO_{2l}.
$$

\begin{lemma}\label{6.3}  \
\begin{enumerate}
    \item For $m$ large enough and depending only on $f$ and $f'$,
    $$
    B_m(t_{l-1}(a),f)-B_m(t_{l-1}(a),f')=0,
    $$ for any $a\in\GL_{l-1}.$
    \item Suppose that $k\leq l-1.$ For any $i\geq k+1$, $P_{k, 1^{l-k}}\tilde{w}_k P_{k, 1^{l-k}}\cap\Omega_{\tilde{w}_i}=\emptyset.$ Thus, 
    $$
    B_m(g,f_{\tilde{w}_i})=0,
    $$
    for any $g\in P_{k, 1^{l-k}}\tilde{w}_k P_{k, 1^{l-k}}$ and $i\geq k+1.$ Hence
    $$
     B_m(g,f)-B_m(g,f')=B_m(g,f_{\tilde{w}_k}),
    $$
    for any $g\in P_{k, 1^{l-k}}\tilde{w}_k P_{k, 1^{l-k}}$ and $m$ large enough.
    \item Let $k\leq l-1.$ For $m$ large enough depending only on $f_{\tilde{w}_k}$ (and hence only on $f$ and $f'$), we have
    $$
    B_m(t_k(a)\tilde{w}_k u_0, f_{\tilde{w}_k})=0,
    $$
    for all $a\in\GL_k$ and $u_0\in N_k\backslash (U_m\cap N_k).$
\end{enumerate}
\end{lemma}

\begin{proof}
\begin{enumerate}
    \item This follows from the proof of \cite[Proposition 4.8]{Haz23a}.
    \item For contradiction, suppose there exists $w\in \mathrm{B}(\SO_{2l})$ such that $$C(w)\subseteq P_{k, 1^{l-k}}\tilde{w}_k P_{k, 1^{l-k}}\cap\Omega_{\tilde{w}_k}.$$ By Lemma \ref{P_{k, 1^{l-k}}}, $$w\in [\tilde{w}_k, w_l^{L_{k, 1^{l-k}}}\tilde{w}_k]\cap\mathrm{B}(\SO_{2l})=\mathrm{B}_k(\SO_{2l}).$$ By \cite[Proposition 4.5]{Haz23a}, we have $$\{\alpha_{k+1},\dots,\alpha_l\}\subseteq\theta_w\subseteq\Delta(\SO_{2l})\backslash\{\alpha_k\}.$$ However, since $w\in\Omega_{\tilde{w}_i}$ $w\geq\tilde{w}_i$ and by Lemma \ref{CPSSh}, $$\theta_w\subseteq\theta_{\tilde{w}_i}=\Delta(\SO_{2l})\backslash\{\alpha_{i}\}.$$
    Since $i\geq k+1$ and $\{\alpha_{k+1},\dots,\alpha_l\}\subseteq\theta_w\subseteq\Delta(\SO_{2l})\backslash\{\alpha_{i}\},$ we have a contradiction.
    \item This proof is the similar to \cite[Lemma 6.3(3)]{Zha18} but we include it for completeness. By Proposition \ref{Omega_w}, $P_{k, 1^{l-k}}\tilde{w}_k P_{k, 1^{l-k}}$ is closed in $\Omega_{\tilde{w}_k}.$ Since $f_{\tilde{w}_k}\in C_c^\infty(\Omega_{\tilde{w}_k},\omega)$, there exists compacts subsets $P'\subseteq P_{k, 1^{l-k}}$ and $N'\subseteq N_k$ such that if $f_{\tilde{w}_k}(p\tilde{w}_k n)\neq 0$, then $p\in P'$ and $n\in N'.$ 
    
    Choose $m$ large enough such that $N'\subseteq U_m\cap N_k.$ This choice depends only on $N'$ and hence $f_{\tilde{w}_k}.$ Let $u\in U_m$ and write $u=u^+ u^-$ where $u^+\in S_k^+\cap U_m$ and $u^-\in N_k\cap U_m.$ Let $u_0\in N_k\backslash (U_m\cap N_k).$ By Lemma \ref{6.2}, $u_0':=(u^+)\inv u_0 u^+\in N_k\backslash (U_m\cap N_k).$ Then, for $a\in\GL_k$ and $u'\in U_{\SO_{2l}}$, we have
    \begin{align*}
    f_{\tilde{w}_k}(u't_k(a)\tilde{w}_k u_0 u)
    =f_{\tilde{w}_k}(u't_k(a)\tilde{w}_k u^+ u_0' u^-) =f_{\tilde{w}_k}(u't_k(a)\tilde{w}_k u^+\tilde{w}_k\inv \tilde{w}_k u_0' u^-).
    \end{align*}
    We have $u't_k(a)\tilde{w}_k u^+\tilde{w}_k\inv\in P_{k, 1^{l-k}}.$ Thus if $ f_{\tilde{w}_k}(u't_k(a)\tilde{w}_k u_0 u)\neq 0,$ then $u't_k(a)\tilde{w}_k u^+\tilde{w}_k\inv\in P'$ and $u_0' u^-\in N'\subseteq U_m\cap N_k.$ Since $u^-\in N_k\cap U_m,$ we must have $u_0'\in N_k\cap U_m.$ But, $u_0'\in N_k\backslash (U_m\cap N_k).$ Therefore, for any $u'\in U_{\SO_{2l}}$ and $u\in U_m,$ $f_{\tilde{w}_k}(u't_k(a)\tilde{w}_k u_0 u)=0.$ Finally, by definition, we find that
    $$
    B_m(t_k(a)\tilde{w}_k u_0, f_{\tilde{w}_k})=0,
    $$
    for all $a\in\GL_k$ and $u_0\in N_k\backslash (U_m\cap N_k).$
    \end{enumerate}
\end{proof}

The next proposition is used to connect the vanishing of certain zeta integrals to information about the partial Bessel functions.

\begin{prop}[{\cite[Lemma 2.3]{Cha19}}]\label{JS Prop}
Let $W'$ be a smooth function on $\GL_n$ such that $W'(ug)=\psi(u)W'(g)$ for any $u\in U_{\GL_n}$ and $g\in\GL_n.$ Suppose that for any irreducible generic representation $\tau$ of $\GL_n$ and for any Whittaker function $W\in\mathcal{W}(\tau,\psi\inv),$ the integral 
$$
\int_{U_{\GL_n} \backslash \GL_n} W'(g)W(g)|\mathrm{det}(g)|^{-s-k}=0,
$$
for $\mathrm{Re}(s)<<0.$ Then, $W'=0.$
\end{prop}

The above proposition follows from \cite[Lemma 3.2]{JS85}. The proof of the above form follows \cite[Corollary 2.1]{Che06} (see the remark after \cite[Lemma 2.3]{Cha19}).

The next proposition is the main step in proving Theorem \ref{l-2 theorem}.

\begin{prop}\label{l-2 prop}
Let $1\leq k \leq l-2.$ The condition $\mathcal{C}(k)$ implies
$$
B_m(t_k(a)\tilde{w}_k,f_{\tilde{w}_k})=0,
$$
for any $a\in\GL_k$ and $m$ large enough depending only on $f$ and $f'.$
\end{prop}

\begin{proof}
Let $m$ be large enough such that the induction hypothesis and Lemma \ref{6.3} hold. Let $(\tau, V_\tau)$ be an irreducible $\psi\inv$-generic representation of $\GL_k$ and $v\in V_\tau.$ Since $k\leq l-2 <l,$ we are embedding $\SO_{2k+1}$ into $\SO_{2l}$ as in $\S$\ref{gps and reps}. Let $j:\SO_{2k+1}\rightarrow\SO_{2l}$ denote the embedding. Recall that $Q_k=L_k V_k$ is the standard Siegel parabolic subgroup of $\SO_{2k+1}$ with Levi subgroup $L_k\cong\GL_k.$ Let $V_{k,m}=\{u\in V_k \, | w^{l,k} j(u) (w^{l,k})\inv\in H_m\}.$ Then $V_{k,m}$ is an open compact subgroup of $V_k.$ Let $i$ be a positive integer with $i\geq\mathrm{max}\{m,i_0(v),i(V_{k,m},v)\}$ (for notation, see Lemma \ref{5.3} for $i_0(v)$ and Lemma \ref{5.4} for $i(V_{k,m},v)$ where we take $D=V_{k,m}$). We consider the sections $\xi_s^{i,v}\in I(\tau,s,\psi\inv)$ and $\tilde{\xi}_{1-s}^{i,v}=M(\tau,s,\psi\inv)\xi_s^{i,v}\in I(\tau^*,1-s,\psi\inv)$ defined in $\S\ref{sections}.$ 

First, we compute the non-intertwined zeta integrals $\Psi(W_m^f,\xi_s^{i,v}).$ We have that $V_k L_k \ol{V}_k$ is dense in $\SO_{2k+1}$ where $\ol{V}_k$ denotes the opposite unipotent radical to $V_k.$ Hence, we can replace the integral over $U_{\SO_{2k+1}}\backslash\SO_{2k+1}$ with an integral over $$U_{\SO_{2k+1}}\backslash V_k L_k \ol{V}_k\cong U_{\GL_k}\backslash\GL_k \times \ol{V}_k,$$
 where $\GL_k$ is realized in $L_k\subseteq\SO_{2k+1}.$ Note for $g=u_1 m_k(a) \ol{u}_2\in U_{\GL_k}\backslash\GL_k \times \ol{V}_k$ we use the quotient Haar measure $dg=|\mathrm{det}(a)|^{-k} d\ol{u}_2 da.$

Thus, we have
\begin{align*}    
&\,\Psi(W_m^f,\xi_s^{i,v})\\
=&\,\int_{U_{\GL_k}\backslash\GL_k}\int_{\ol{V}_k}\left(\int_{r\in R^{l,k}} W_m^f(r w^{l,k} j(l_k(a)\ol{u}) (w^{l,k})\inv)dr\right)   \xi_s^{i,v}(l_k(a)\ol{u}, I_{l})
    |\mathrm{det}(a)|^{-k}d\ol{u}da.
\end{align*}
By Equation \eqref{xi nonintertwined}, we have
\begin{align*}
   &\,\Psi(W_m^f,\xi_s^{i,v})\\
=&\,\int_{U_{\GL_k}\backslash\GL_k}\int_{\ol{V}_{k,i}}\left(\int_{r\in R^{l,k}} W_m^f(r w^{l,k} j(l_k(a)\ol{u}) (w^{l,k})\inv)dr\right)
   |\mathrm{det}(a)|^{s-\frac{1}{2}-k} W_v(a)d\ol{u}da. 
\end{align*}
Furthermore, for $\ol{u}\in\ol{V}_{k,i}$, we have $w^{l,k} j(\ol{u}) (w^{l,k})\inv\in H_m.$ By Proposition \ref{partialBesselprop},
$$
W^f_m(g w^{l,k} j(\ol{u}) (w^{l,k})\inv) = W_m^f(g),
$$
for any $g\in\SO_{2l}.$ Thus,
\begin{align*}
  &\,\Psi(W_m^f,\xi_s^{i,v})\\
  =&\,\mathrm{Vol}(\ol{V}_{k,i})\int_{U_{\GL_k}\backslash\GL_k}\left(\int_{r\in R^{l,k}} W_m^f(r w^{l,k} j(l_k(a)) (w^{l,k})\inv)dr\right)
  |\mathrm{det}(a)|^{s-\frac{1}{2}-k} W_v(a)da.  
\end{align*}
Let $q_k(a)=\mathrm{diag}(I_{l-k-1},a,I_2,a^*,I_{l-k-1}).$ We have $j(l_k(a))=q_k(a).$ Thus, 
$$w^{l,k} q_k(a) (w^{l,k})\inv=\mathrm{diag}(a, I_{2l-2k}, a^*).$$
Let $$r_x=\left(\begin{matrix}
I_k & & & & \\
x & I_{l-k-1} & & & \\
& & I_2 & & \\
& & & I_{l-k-1} & \\
& & & x^\prime & I_k
\end{matrix}\right).$$ Then, $$
r_x w^{l,k} q_k(a) (w^{l,k})\inv
=\left(\begin{matrix}
a & & & & \\
xa & I_{l-k-1} & & & \\
& & I_2 & & \\
& & & I_{l-k-1} & \\
& & & x^\prime a^* & a^*
\end{matrix}\right)\in L_l.
$$
Thus, by Lemma \ref{6.3}(1),
$$
W_m^f(r w^{l,k} j(l_k(a)) (w^{l,k})\inv)=W_m^{f'}(r w^{l,k} j(l_k(a)) (w^{l,k})\inv),
$$
and hence 
$$
\Psi(W_m^f,\xi_s^{i,v})=\Psi(W_m^{f'},\xi_s^{i,v}).
$$

By assumption, $\gamma(s,\pi\times\tau,\psi)=\gamma(s,\pi'\times\tau,\psi).$ Since we have that $\Psi(W_m^f,\xi_s^{i,v})=\Psi(W_m^{f'},\xi_s^{i,v}),$ the definition of the $\gamma$-factor gives
$$\Psi(W_m^f,\tilde{\xi}_{1-s}^{i,v})=\Psi(W_m^{f'},\tilde{\xi}_{1-s}^{i,v}).$$
$V_k L_k w_k V_k$ is dense in $\SO_{2k+1}.$ Hence, we can replace the integral over $U_{\SO_{2k+1}}\backslash
\SO_{2k+1}$ with an integral over $U_{\SO_{2k+1}}\backslash V_k L_k w_k V_k \cong U_{\GL_k}\backslash\GL_k w_k V_k.$ We obtain $\Psi(W_m^f,\tilde{\xi}_{1-s}^{i,v})$ equals
\begin{align*}
&\,\int_{U_{\GL_k}\backslash\GL_k}\int_{V_k}\left(\int_{r\in R^{l,k}} W_m^f(r w^{l,k} j(l_k(a)w_k u) (w^{l,k})\inv)dr\right)  \\
&\,\cdot \tilde{\xi}_{1-s}^{i,v}(l_k(a)w_ku, I_{l})
  |\mathrm{det}(a)|^{-k}duda,  
\end{align*}
and we have a similar expression for $\Psi(W_m^{f'},\tilde{\xi}_{1-s}^{i,v}).$

We have $w^{l,k} j(w_k) (w^{l,k})\inv=\tilde{w}_k$ and 
$$w^{l,k} j(l_k(a)) (w^{l,k})\inv=\mathrm{diag}(a,I_{2l-2k},a^*)=t_k(a).$$ Again, let
$$r_x=\left(\begin{matrix}
I_k & & & & \\
x & I_{l-k-1} & & & \\
& & I_2 & & \\
& & & I_{l-k-1} & \\
& & & x^\prime & I_k
\end{matrix}\right).$$ Then, $
r_x t_k(a)=t_k(a)r_{ax}
$ and 
$$r_{ax}\tilde{w}_k=\tilde{w}_k\left(\begin{matrix}
I_k & & & x'a^* & \\
& I_{l-k-1} & & & xa \\
& & I_2 & & \\
& & & I_{l-k-1} & \\
& & & & I_k
\end{matrix}\right).$$ 
Also, for 
$$u=\left(\begin{matrix}
I_k & y_1   & y \\
    & 1     & y_1' \\
    &       & I_k
\end{matrix}\right)\in V_k,$$ we have 
$$
w^{l,k} j(u) (w^{l,k})\inv= \left(\begin{matrix}
I_k & 0 & * & 0 & * \\
& I_{l-k-1} & 0 & 0 & 0  \\
& & I_2 & 0 & * \\
& & & I_{l-k-1} & 0 \\
& & & & I_k
\end{matrix}\right).
$$
Recall $P_k$ is the standard maximal parabolic subgroup of $\SO_{2l}$ with Levi subgroup $M_k$ isomorphic to $\GL_k\times\SO_{2(l-k)}$ and unipotent radical $N_k.$ We have $$j(a,x,u):=\left(\begin{matrix}
I_k & & & x'a^* & \\
& I_{l-k-1} & & & xa \\
& & I_2 & & \\
& & & I_{l-k-1} & \\
& & & & I_k
\end{matrix}\right)w^{l,k} j(u) (w^{l,k})\inv\in N_k.$$ 
Thus, we have shown that
$$
r_x w^{l,k} j(l_k(a)w_k u) (w^{l,k})\inv 
=t_k(a)\tilde{w}_k j(a,x,u)\inv\in P_k \tilde{w}_k P_k.
$$
From Lemma \ref{6.3}(2), we obtain
\begin{align*}
&\,W_m^f(t_k(a)\tilde{w}_k j(a,x,u))-W_m^{f'}(t_k(a)\tilde{w}_k j(a,x,u)) \\
=&\,B_m(t_k(a)\tilde{w}_k j(a,x,u),f)-B_m(t_k(a)\tilde{w}_k j(a,x,u),f') \\
=&\,B_m(t_k(a)\tilde{w}_k j(a,x,u),f_{\tilde{w}_k}).
\end{align*}

Altogether, we  have
\begin{align}\label{l-2 twist eqn}
\begin{split}
0 =&\,\Psi(W_m^f,\tilde{\xi}_{1-s}^{i,v})-\Psi(W_m^{f'},\tilde{\xi}_{1-s}^{i,v}) \\
=&\,\int_{U_{\GL_k}\backslash\GL_k}\int_{V_k}\left(\int_{r_x\in R^{l,k}} B_m(t_k(a)\tilde{w}_k j(a,x,u),f_{\tilde{w}_k})dr\right) \\
&\,\cdot \tilde{\xi}_{1-s}^{i,v}(l_k(a)w_ku, I_{l})
|\mathrm{det}(a)|^{-k}duda \\
=&\,\int_{U_{\GL_k}\backslash\GL_k}\int_{V_k}\left(\int_{r_x\in R^{l,k}} B_m(t_k(a)\tilde{w}_k j(I_k,x,u),f_{\tilde{w}_k})dr\right) \\
&\,\cdot \tilde{\xi}_{1-s}^{i,v}(l_k(a)w_ku, I_{l})
|\mathrm{det}(a)|^{-l}duda,
\end{split}
\end{align}
where we made the change of variables $x\mapsto xa\inv$ in the last step.

Next, let $D_m=\{j(I_k,x,u)\, | \, j(I_k,x,u)\in H_m\}.$ By Lemma \ref{6.3}(3), we have, for $j(I_k,x,u)\notin D_m,$
$$
B_m(t_k(a)\tilde{w}_k j(I_k,x,u),f_{\tilde{w}_k})=0.
$$
However, by Proposition \ref{partialBesselprop}, for $j(I_k,x,u)\in D_m,$ 
$$
B_m(t_k(a)\tilde{w}_k j(I_k,x,u),f_{\tilde{w}_k})=B_m(t_k(a)\tilde{w}_k,f_{\tilde{w}_k}).
$$
Furthermore, for $j(I_k,x,u)\in D_m$, we have $u\in V_{k,m}$ and thus, by Lemma \ref{5.4}, $\tilde{\xi}_{1-s}^{i,v}(l_k(a) w_n u)=\mathrm{Vol}(\ol{V}_{k,i})|\mathrm{det}(a)|^{\frac{1}{2}-s} W_v^*(a).$

From Equation \eqref{l-2 twist eqn}, we obtain
$$
0=\int_{U_{\GL_k}\backslash\GL_k}B_m(t_k(a)\tilde{w}_k,f_{\tilde{w}_k}) |\mathrm{det}(a)|^{\frac{1}{2}-s-l}W_v^*(a)da.
$$
This holds for any irreducible $\psi\inv$-generic representations $(\tau, V_\tau)$ of $\GL_k$ and $v\in V_\tau.$ From Proposition \ref{JS Prop}, we have
$$
B_m(t_k(a)\tilde{w}_k,f_{\tilde{w}_k})=0,
$$
for any $a\in\GL_k.$
\end{proof}

As a corollary of the computations above, we obtain the following result on the equality of twisted gamma factors of $\pi$ and $\pi^c$, which can also be implied by the work of \cite{Art13, JL14}. 

\begin{cor}\label{conj gamma1}
Let $\pi$ be an irreducible  $\psi$-generic supercuspidal representation of $\SO_{2l}$. Then $\gamma(s, \pi\times\tau,\psi)=\gamma(s, \pi^c\times\tau,\psi)$ for all irreducible $\psi^{-1}$-generic representations $\tau$ of $\GL_{k}$ for $k\leq l-2$.
\end{cor}

\begin{proof}
Since $c\tilde{t}\inv t_k(a)\tilde{w}_k \tilde{t}c=t_k(a)\tilde{w}_k$ for any $a\in\GL_k$, repeating the above arguments for $\pi'=\pi^c$ without the assumption that the $\gamma$-factors are equal gives the claim.
\end{proof}

Set $\psi_n=1$ for any $n\leq l-2$ and $\psi_{l-1}=\frac{1}{2}.$ These are the coefficients which are used to define the generic character $\psi$ of $U_{\SO_{2l}}.$ Recall that $S_{\SO_{2l}}$ denotes the split part of the torus $T_{\SO_{2l}}.$ For $w\in W(\SO_{2l}),$ we define
$$T_w=\{t\in S_{\SO_{2l}} \, | \, \psi_{n'} \alpha_{n'}(t)=\psi_{n} \, \mathrm{for \, any} \, \alpha_n\in\theta_w \, \mathrm{such \, that}\, w\alpha_{n}=\alpha_{n'} \}.$$
We are ready to prove Theorem \ref{l-2 theorem}.
\begin{proof}[Proof of Theorem \ref{l-2 theorem}]

Let $w_{\mathrm{max}}=w_l^{L_{k,1^{l-k}}}\tilde{w}_k.$ Then $w_{long}w_{\mathrm{max}}$ is the long Weyl element of $M_{w_{\mathrm{max}}}\cong\GL_{1}^k\times\SO_{2(l-k)}.$ Let $Z(\SO_{2l})$ denote the center of $\SO_{2l}.$ We have 
$$
T_{w_{\mathrm{max}}}=\left\{t_k(a_1,\dots,a_k) \, | \, a_i\in F^\times \right\} \times Z(\SO_{2l}).
$$
For $t\in T_{w_{\mathrm{max}}},$ write $t=zt_k(a)$ where $z\in Z(\SO_{2l})$ and $a\in\GL_1^k.$ We have \[t_k(a)w_l^{L_{k,1^{l-k}}}=t_k(a)w_l^{L_{k,1^{l-k}}}=t_n(b)\] for some $b\in\GL_k.$ Thus, for any $t\in T_{w_{\mathrm{max}}},$
$$
B_m(tw_{\mathrm{max}},f_{\tilde{w}_k})=\omega(z)B_m(t_k(a)w_l^{L_{k,1^{l-k}}}\tilde{w}_k,f_{\tilde{w}_k})=0,
$$
by Proposition \ref{l-2 prop}.

Next, suppose that $w\in \mathrm{B}(\SO_{2l})$ with $\tilde{w}_k\leq w \leq w_{\mathrm{max}}.$ By Lemma \ref{CPSSh}, $A_w\subseteq A_{w_{\mathrm{max}}}.$ Also $[\tilde{w}_k, w_{\mathrm{max}}]\cap\mathrm{B}(\SO_{2l})=\mathrm{B}_k(\SO_{2l})$. Again, from Proposition \ref{l-2 prop}, we find
$$
B_m(tw,f_{\tilde{w}_k})=0,
$$
for any $t\in T_w.$

Let 
$$
\Omega_{\tilde{w}_k}'=\bigcup_{\substack{w\in \mathrm{B}(\SO_{2l}), w>w_{\mathrm{max}}, \\ d_{B}(w,w_{\mathrm{max}})=1}} \Omega_w.
$$
By Theorem \ref{CST}, there exists $f_{\tilde{w}_k}'\in C_c^\infty(\Omega_{\tilde{w}_k}',\omega)$ such that
$$
B_m(g, f_{\tilde{w}_k})=B_m(g,f_{\tilde{w}_k}'),
$$
for any $g\in\SO_{2l}$ and $m$ large enough (depending only $f_{\tilde{w}_k}$ and hence $f$ and $f'$).

We describe the set $\{w\in \mathrm{B}(\SO_{2l}) \, | \, w>w_{\mathrm{max}},  d_{B}(w,w_{\mathrm{max}})=1\}.$ From the proof of Lemma \ref{6.3}(2), we have $\theta_{w_{\mathrm{max}}}=\{\alpha_{k+1},\dots,\alpha_{l-1}\}.$ Since $w>w_{\mathrm{max}}$ and  $d_{B}(w,w_{\mathrm{max}})=1,$ we must have $\theta_w=\theta_{w_{\mathrm{max}}}\backslash\{\alpha_i\}$ for some $k+1\leq i \leq {l-1}.$ Set $w_i'=w_{\theta_{w_{\mathrm{max}}}\backslash\{\alpha_i\}}.$ Then,
$$
\{w\in \mathrm{B}(\SO_{2l}) \, | \, w>w_{\mathrm{max}},  d_{B}(w,w_{\mathrm{max}})=1\}=\{w_i' \, | \, k+1\leq i \leq {l-1} \}.
$$
Thus, 
$$\Omega_{\tilde{w}_k}'=\bigcup_{i=k+1}^{l-1} \Omega_{w_i'}.$$ By a partition of unity argument (similar to the proof of Corollary \ref{C(0)}), there exists $f_{w_i'}\in C_c^\infty(\Omega_{w_i'},\omega)$ for $k+1\leq i \leq l-1$ such that
$$
f_{\tilde{w}_k}'=\sum_{i=k+1}^{l-1} f_{w_i'}.
$$
So,
\begin{equation}\label{inductive step eqn}
B_m(g,f_{\tilde{w}_k})=B_m(g,f_{\tilde{w}_k}')=\sum_{i=k+1}^{l-1} B_m(g,f_{w_i'}),
\end{equation}
for any $g\in\SO_{2l}$ and $m$ large enough.

Finally, fix $i$ such that $k+1\leq i \leq {l-1}.$ We have $\theta_{\tilde{w}_i}=\Delta(\SO_{2l})\backslash\{\alpha_i\}$. Since $\theta_{w_i'}\subseteq\theta_{w_{\mathrm{max}}}\backslash\{\alpha_i\}\subseteq\theta_{\tilde{w}_i},$ it follows that $w_i'>\tilde{w}_i$ and hence $\Omega_{w_i'}\subseteq\Omega_{\tilde{w}_i}.$ By Proposition \ref{Omega_w}, $\Omega_{w_i'}$ is open in $\Omega_{\tilde{w}_i}.$ Hence, $C_c^\infty(\Omega_{w_i'},\omega)\subseteq C_c^\infty(\Omega_{\tilde{w}_i},\omega).$ Thus $f_{w_i}'\in C_c^\infty(\Omega_{w_i'},\omega)$. Set
$$
f_{\tilde{w}_i}'= f_{w_i}' + f_{\tilde{w}_i}\in C_c^\infty(\Omega_{w_i'},\omega).
$$
Then, by the inductive hypothesis and Equation \eqref{inductive step eqn}, we have
$$
B_m(g,f)-B_m(g,f')=\sum_{i=k+1}^{l-1} B_m(g,f_{\tilde{w}_i}').
$$
for any $g\in\SO_{2l}$ and $m$ large enough depending only on $f$ and $f'.$
This completes the proof of Theorem \ref{l-2 theorem}.
\end{proof}

\subsection{\texorpdfstring{Twists up to $\GL_{l-1}$}{}}

From Theorem \ref{l-2 theorem}, we have that $\mathcal{C}(l-2)$ implies
\begin{equation}\label{Bessel difference at l-2}
    B_m(g,f)-B_m(g,f')= B_m(g,f_{\tilde{w}_{l-1}}),
\end{equation}
for any $g\in\SO_{2l}$ and $m$ large enough. In this section, we show that $\mathcal{C}({l-1})$ implies $B_m(g,f_{\tilde{w}_{l-1}})$ vanishes
if $g$ is fixed by the outer conjugation by $c.$ This does not determine the entirety of the support though. Indeed, the support may contain elements which are not fixed by $c$, e.g., $t\tilde{w}_{l-1}$ where $t\in T_{\SO_{2l}}\setminus S_{\SO_{2l}}.$ These elements will be determined by the condition $\mathcal{C}({l})$ in the next subsection at the cost of a conjugation of the partial Bessel function. We proceed with the computation for $\mathcal{C}({l-1})$.

\begin{thm}\label{l-1 thm}
The condition $\mathcal{C}(l-1)$ implies
$$
B_m(t_{l-1}(a)\tilde{w}_{l-1},f_{\tilde{w}_{l-1}})=0,
$$
for any $a\in\GL_{l-1}$ and $m$ large enough depending only on $f$ and $f'.$
\end{thm}

\begin{proof}
Let $m$ be large enough such that Theorem \ref{l-2 theorem} and Lemma \ref{6.3} hold. Let $(\tau, V_\tau)$ be an irreducible $\psi\inv$-generic representation of $\GL_{l-1}$ and $v\in V_\tau.$ We consider the embedding of $\SO_{2l-1}$ into $\SO_{2l}$ as in $\S$\ref{gps and reps}. Let $j:\SO_{2l-1}\rightarrow\SO_{2l}$ denote the embedding. Recall that $Q_{l-1}=L_{l-1} V_{l-1}$ is the standard Siegel parabolic subgroup of $\SO_{2l-1}$ with Levi subgroup $L_{l-1}\cong\GL_{l-1}.$ Note that $w^{l,l-1}$ is trivial. Let $V_{{l-1},m}=\{u\in V_{l-1} \, |  j(u) \in H_m\}.$ Then $V_{{l-1},m}$ is an open compact subgroup of $V_{l-1}.$ Let $i$ be a positive integer with $i\geq\mathrm{max}\{m,i_0(v),i(V_{{l-1},m},v)\}$ (for notation, see Lemma \ref{5.3} for $i_0(v)$ and Lemma \ref{5.4} for $i(V_{{l-1},m},v)$ where we take $D=V_{{l-1},m}$). We fix the section $\xi_s^{i,v}\in I(\tau,s,\psi\inv)$ defined in $\S\ref{sections}.$ We also consider the section $\tilde{\xi}_{1-s}^{i,v}=M(\tau,s,\psi\inv)\xi_s^{i,v}\in I(\tau^*,1-s,\psi\inv)$. 

First, we compute the non-intertwined zeta integrals $\Psi(W_m^f,\xi_s^{i,v}).$ We have that $V_{l-1} L_{l-1} \ol{V}_{l-1}$ is dense in $\SO_{2l-1}$ where $\ol{V}_{l-1}$ denotes the opposite unipotent radical to $V_{l-1}.$ Hence, we can replace the integral over $U_{\SO_{2l-1}}\backslash\SO_{2l-1}$ with an integral over $U_{\SO_{2l-1}}\backslash V_{l-1} L_{l-1} \ol{V}_{l-1}\cong U_{\GL_{l-1}}\backslash\GL_{l-1} \times \ol{V}_{l-1}$ where $\GL_{l-1}$ is realized in $L_{l-1}\subseteq\SO_{2l-1}.$ For $g=u_1 l_{l-1}(a) \ol{u}_2\in U_{\GL_k}\backslash\GL_{l-1} \times \ol{V}_{l-1}$ we use the quotient Haar measure $dg=|\mathrm{det}(a)|^{-({l-1})} d\ol{u}_2 da.$ 

Note that $R^{l,l-1}$ is trivial. We have
\begin{align*}
    \Psi(W_m^f,\xi_s^{i,v})=\int_{U_{\GL_{l-1}}\backslash\GL_{l-1}}\int_{\ol{V}_{l-1}} W_m^f(  j(l_{l-1}(a)\ol{u})) \xi_s^{i,v}(l_{l-1}(a)\ol{u}, I_{l})|\mathrm{det}(a)|^{-({l-1})}d\ol{u}da.
\end{align*}
By Equation \eqref{xi nonintertwined}, we have
$$
\Psi(W_m^f,\xi_s^{i,v})=\int_{U_{\GL_{l-1}}\backslash\GL_{l-1}}\int_{\ol{V}_{{l-1},i}}W_m^f( j(l_{l-1}(a)\ol{u})) |\mathrm{det}(a)|^{s-\frac{1}{2}-({l-1})} W_v(a)d\ol{u}da.
$$
Furthermore, for $\ol{u}\in\ol{V}_{{l-1},i}$, we have $j(\ol{u}) \in H_m.$ By Proposition \ref{partialBesselprop}, for any $g\in\SO_{2l},$
$$
W^f_m(g j(\ol{u}))= W_m^f(g).
$$
Thus,
$$
\Psi(W_m^f,\xi_s^{i,v})=\mathrm{Vol}(\ol{V}_{{l-1},i})\int_{U_{\GL_{l-1}}\backslash\GL_{l-1}}W_m^f( j(l_{l-1}(a))  |\mathrm{det}(a)|^{s-\frac{1}{2}-({l-1})} W_v(a)da.
$$
Let $q_{l-1}(a)=\mathrm{diag}(a,I_2,a^*).$ We have $j(l_{l-1}(a))=q_{l-1}(a).$ 
Thus, by Lemma \ref{6.3}(1),
$$
W_m^f( j(l_{l-1}(a)) )=W_m^{f'}( j(l_{l-1}(a)) ),
$$
and hence 
$$
\Psi(W_m^f,\xi_s^{i,v})=\Psi(W_m^{f'},\xi_s^{i,v}).
$$

By assumption, $\gamma(s,\pi\times\tau,\psi)=\gamma(s,\pi'\times\tau,\psi).$ Since $\Psi(W_m^f,\xi_s^{i,v})=\Psi(W_m^{f'},\xi_s^{i,v}),$ the local functional equation gives
$$\Psi(W_m^f,\tilde{\xi}_{1-s}^{i,v})=\Psi(W_m^{f'},\tilde{\xi}_{1-s}^{i,v}).$$

We have that $V_{l-1} L_{l-1} w_{l-1} V_{l-1}$ is dense in $\SO_{2l-1}.$ Hence, we can replace the integral over $U_{\SO_{2l-1}}\backslash\SO_{2l-1}$ with an integral over $U_{\SO_{2l-1}}\backslash V_{l-1} L_{l-1} w_{l-1} V_{l-1}$ which is isomorphic to $U_{\GL_{l-1}}\backslash\GL_{l-1} w_{l-1} V_{l-1}.$ We obtain
\begin{align*}
   &\,\Psi(W_m^f,\tilde{\xi}_{1-s}^{i,v})\\
   =&\, \int_{U_{\GL_{l-1}}\backslash\GL_{l-1}}\int_{V_{l-1}} W_m^f(  j(l_{l-1}(a)w_{l-1} u))    \tilde{\xi}_{1-s}^{i,v}(l_{l-1}(a)w_{l-1} u, I_{l})|\mathrm{det}(a)|^{-({l-1})}duda. 
\end{align*}
We have a similar expression for $\Psi(W_m^{f'},\tilde{\xi}_{1-s}^{i,v}).$

Note that $j(w_{l-1})=\tilde{w}_{l-1}$ and $ j(l_{l-1}(a)))=\mathrm{diag}(a,I_{2},a^*)=t_{l-1}(a).$ For 
$$u=\left(\begin{matrix}
I_{l-1} & y_1   & y \\
    & 1     & y_1' \\
    &       & I_{l-1}
\end{matrix}\right)\in V_{l-1},$$ we have 
$$
 j(u) = \left(\begin{matrix}
I_{l-1} & y_1 &  & y \\
&  1 &  & y_1' \\
& & 1 & \\
& & &  I_{l-1}
\end{matrix}\right).
$$
Recall $P_{l-1}$ is the standard maximal parabolic subgroup of $\SO_{2l}$ with Levi subgroup $M_{l-1}$ isomorphic to $\GL_{l-1}\times\SO_{2}$ and unipotent radical $N_{l-1}.$
We have
$$
 j(l_{l-1}(a)w_{l-1} u)  
=t_{l-1}(a)\tilde{w}_{l-1} j(u)\inv\in P_{l-1} \tilde{w}_{l-1} P_{l-1}.
$$
Now, unlike the split case, we can apply Lemma \ref{6.3}(2).
\begin{align*}
&\,W_m^f(t_{l-1}(a)\tilde{w}_{l-1} j(u))-W_m^{f'}(t_{l-1}(a)\tilde{w}_{l-1} j(u))
\\
=&\,B_m(t_{l-1}(a)\tilde{w}_{l-1} j(u),f)-B_m(t_{l-1}(a)\tilde{w}_{l-1} j(u),f')
\\
=&\,B_m(t_{l-1}(a)\tilde{w}_{l-1} j(u),f_{\tilde{w}_{l-1}}).
\end{align*}

Altogether, we  have
\begin{align}\label{l-1 twist eqn}
\begin{split}
0&=\Psi(W_m^f,\tilde{\xi}_{1-s}^{i,v})-\Psi(W_m^{f'},\tilde{\xi}_{1-s}^{i,v}) \\
&=\int_{U_{\GL_{l-1}}\backslash\GL_{l-1}}\int_{V_{l-1}} B_m(t_{l-1}(a)\tilde{w}_{l-1} j(u),f_{\tilde{w}_{l-1}})  \tilde{\xi}_{1-s}^{i,v}(l_{l-1}(a)w_{l-1} u, I_{l})|\mathrm{det}(a)|^{-({l-1})}duda. 
\end{split}
\end{align}

By Lemma \ref{6.3}(3), for $u\notin V_{l-1,m},$
$$
B_m(t_{l-1}(a)\tilde{w}_{l-1} j(u),f_{\tilde{w}_{l-1}})=0.
$$
However, by Proposition \ref{partialBesselprop}, for $u\in V_{l-1,m},$ 
$$
B_m(t_{l-1}(a)\tilde{w}_{l-1} j(u),f_{\tilde{w}_{l-1}})=B_m(t_{l-1}(a)\tilde{w}_{l-1},f_{\tilde{w}_{l-1}}).
$$
Note that we used $\psi_m(j(u))=1.$
By Equation \eqref{xi intertwined}, for $u\in V_{l-1,m}$, $$\tilde{\xi}_{1-s}^{i,v}(l_{l-1}(a) w_{l-1} u, I_l)=\mathrm{Vol}(\ol{V}_{{l-1},i})|\mathrm{det}(a)|^{\frac{1}{2}-s} W_v^*(a).$$

From Equation \eqref{l-1 twist eqn}, we obtain
$$
0=\int_{U_{\GL_{l-1}}\backslash\GL_{l-1}}B_m(t_{l-1}(a)\tilde{w}_{l-1},f_{\tilde{w}_{l-1}}) |\mathrm{det}(a)|^{\frac{1}{2}-s-(l-1)}W_v^*(a)da.
$$
This holds for any irreducible $\psi\inv$-generic representations $(\tau, V_\tau)$ of $\GL_{l-1}$ and $v\in V_\tau.$ By Proposition \ref{JS Prop}, we have
$$
B_m(t_{l-1}(a)\tilde{w}_{l-1},f_{\tilde{w}_{l-1}})=0,
$$
for any $a\in\GL_{l-1}.$
This completes the proof of Theorem \ref{l-1 thm}.
\end{proof}

Similarly to the other twists, we have the following corollary. 

\begin{cor}\label{conj gamma2}
Let $\pi$ be an irreducible  $\psi$-generic supercuspidal representation of $\SO_{2l}$. Then $\gamma(s, \pi\times\tau,\psi)=\gamma(s, \pi^c\times\tau,\psi)$ for all irreducible $\psi^{-1}$-generic representations $\tau$ of $\GL_{l-1}.$
\end{cor}

\begin{proof}
By Corollary \ref{conj gamma1}, $\gamma(s, \pi\times\tau,\psi)=\gamma(s, \pi^c\times\tau,\psi)$ for all irreducible $\psi^{-1}$-generic representations $\tau$ of $\GL_{k}$ for $k\leq l-2.$ For $k=l-1$, the proof is similar to that of Corollary \ref{conj gamma1}.
\end{proof}

\subsection{\texorpdfstring{Twists up to $\GL_l$}{}}

We consider the twists up to $\GL_l$.
Recall from Equation \eqref{Bessel difference at l-2} that the condition $\mathcal{C}(l-2)$ implies, for any $g\in\SO_{2l},$
$$
B_m(g,f)-B_m(g,f')= B_m(g,f_{\tilde{w}_{l-1}}).
$$
Thus, we have a similar result for the conjugation of the partial Bessel functions 
$$
B_m^c(g,f)-B_m^c(g,f')= B_m^c(g,f_{\tilde{w}_{l-1}}).
$$

The following proposition describes the preimage of $Q_l w_l V_l$ with respect to the embedding of $\SO_{2l}$ into $\SO_{2l+1}.$ This preimage will be needed in computation of the zeta integral. 

\begin{prop}[{\cite[Proposition 7.2]{Haz23a}}]\label{l embed}
Let $w\in\mathrm{B}(\SO_{2l})$ and $$t=\mathrm{diag}(t_1,\dots,t_{l-1},\left(\begin{matrix}
a & b\rho \\
b & a
\end{matrix}\right),t_{l-1}\inv,\dots,t_1\inv) \in T_{\SO_{2l}}.$$ Then $tw\in Q_l w_l V_l$ (via the embedding of $\SO_{2l}$ into $\SO_{2l+1}$) if and only if we have $w\in\mathrm{B}_{l-1}(\SO_{2l})$ and $a\neq 1$. Moreover, if we write $w=t_{l-1}(w')\tilde{w}_{l-1}$ where $w'\in W(\GL_{l-1})$ and $tw=l_l(A)n_1 w_l n_2\in Q_l w_l V_l,$ then $A^*=\left(\begin{matrix}
& \frac{\gamma}{2}(1-a) \\
\mathrm{diag}(t_{l-1}\inv,\dots,t_{1}\inv)(w')^* &
\end{matrix}\right)$.
\end{prop}

{ Our next goal is to define a compact subset which we need to average the zeta integrals over later. To this end, we first explicate the embedding of $\SO_{2l}$ into $\SO_{2l+1}$. Recall we embed $\SO_{2l}$ into $\SO_{2l+1}$ via
$$
g:=\left(\begin{matrix}
A & B \\
C & D
\end{matrix}\right)
\mapsto
M^{-1}
\left(\begin{matrix}
A & & B \\
 & 1 & \\
C & & D  
\end{matrix}\right)
M
$$
where $A, B, C,$ and $D$ are $l\times l$ matrices and
$$
M=\mathrm{diag}(I_{l-1},\left(\begin{matrix}
    0 & 1 & 0 \\
    \frac{1}{2} & 0 & \frac{1}{2\gamma} \\
    \frac{1}{2} & 0 & \frac{-1}{2\gamma} 
    \end{matrix}\right), I_{l-1}).
$$
For brevity, we set 
$$
\tilde{M}:=\left(\begin{matrix}
    0 & 1 & 0 \\
    \frac{1}{2} & 0 & \frac{1}{2\gamma} \\
    \frac{1}{2} & 0 & \frac{-1}{2\gamma} 
    \end{matrix}\right).
$$
We can further decompose the matrices $A,B,C,D$ as follows:
\begin{equation*}
A=\left(\begin{matrix}
A_{1,1} & A_{1,2} \\
A_{2,1} & A_{2,2}
\end{matrix}\right),   B=\left(\begin{matrix}
B_{1,1} & B_{1,2} \\
B_{2,1} & B_{2,2}
\end{matrix}\right), C=\left(\begin{matrix}
C_{1,1} & C_{1,2} \\
C_{2,1} & C_{2,2}
\end{matrix}\right), D=\left(\begin{matrix}
D_{1,1} & D_{1,2} \\
D_{2,1} & D_{2,2}
\end{matrix}\right),
\end{equation*}
where $A_{1,1},B_{1,2},C_{2,1},D_{2,2}$ are $(l-1)\times(l-1)$ matrices, $A_{1,2},B_{1,1},C_{2,2},D_{2,1}$ are $(l-1)\times 1$ matrices, $A_{2,1},B_{2,2},C_{1,1},D_{1,2}$ are $1\times(l-1)$ matrices, and $A_{2,2},B_{2,1},C_{1,2},D_{1,1}$ are $1\times 1$ matrices. Then the embedding of
$g$ in $\SO_{2l+1}$ is given by
\begin{equation}\label{embedding}
    \left(\begin{matrix}
    A_{1,1} & \left(\begin{matrix}
      A_{1,2} & 0 & B_{1,1}  
    \end{matrix}\right)\tilde{M} & B_{1,2} \\
    \tilde{M}\inv\left(\begin{matrix}
      A_{2,1} \\
      0 \\
      C_{1,1}
    \end{matrix}\right) & \tilde{M}\inv\left(\begin{matrix}
        A_{2,2} & 0 & B_{2,1} \\
        0 & 1 & 0 \\
        C_{1,2} & 0 & D_{1,1}
    \end{matrix}\right)\tilde{M} & \tilde{M}\inv\left(\begin{matrix}
        B_{2,2} \\
        0 \\
        D_{1,2}
    \end{matrix}\right) \\
    C_{2,1} & \left(\begin{matrix}
      C_{2,2} & 0 & D_{2,1}  
    \end{matrix}\right)\tilde{M} & D_{2,2} 
    \end{matrix}\right)
\end{equation}
Furthermore, we have
\begin{equation}\label{embedding details}
\begin{split}
    \left(\begin{matrix}
      A_{1,2} & 0 & B_{1,1}  \end{matrix}\right)\tilde{M}&=\left(\begin{matrix}
      \frac{B_{1,1}}{2} & A_{1,2} & \frac{-B_{1,1}}{2\gamma}  
    \end{matrix}\right), \\ \left(\begin{matrix}
      C_{2,2} & 0 & D_{2,1}  \end{matrix}\right)\tilde{M}&=\left(\begin{matrix}
      \frac{D_{2,1}}{2} & C_{2,2} & \frac{-D_{2,1}}{2\gamma}  
    \end{matrix}\right), \\
    \tilde{M}\inv\left(\begin{matrix}
      A_{2,1} \\
      0 \\
      C_{1,1}
    \end{matrix}\right)&=\left(\begin{matrix}
      C_{1,1} \\
      A_{2,1} \\
      -\gamma C_{1,1}
    \end{matrix}\right), \\ \tilde{M}\inv\left(\begin{matrix}
      B_{2,2} \\
      0 \\
      D_{1,2}
    \end{matrix}\right)&=\left(\begin{matrix}
      D_{1,2} \\
      B_{2,2} \\
      -\gamma D_{1,2}
    \end{matrix}\right), \\
    \tilde{M}\inv\left(\begin{matrix}
        A_{2,2} & 0 & B_{2,1} \\
        0 & 1 & 0 \\
        C_{1,2} & 0 & D_{1,1}
\end{matrix}\right)\tilde{M}&=\left(\begin{matrix}
        \frac{D_{1,1}+1}{2} & C_{1,2} & \frac{-D_{1,1}+1}{2\gamma} \\
        \frac{B_{2,1}}{2} & A_{2,2} & \frac{-B_{2,1}}{2\gamma} \\
        \frac{\gamma(1-D_{1,1})}{2} & -\gamma C_{1,2} & \frac{D_{1,1}+1}{2}
    \end{matrix}\right).
    \end{split}
\end{equation}

Let $K_m^{\GL_l}=I_{l}+Mat_{l\times l}(\mathfrak{p}^m)\subseteq\GL_l.$ We identify $K_m^{\GL_l}$ with its image under the isomorphism $l_l:\GL_l\cong L_l\subseteq \SO_{2l+1}.$ Also, let $K_m^{\SO_{2l}}=K_m^{\GL_{2l}}\cap\SO_{2l}$ and $K_m^{\SO_{2l+1}}=K_m^{\GL_{2l+1}}\cap\SO_{2l+1}.$ Let $K_m^{\GL_l}=K^- K^\circ K^+$ be the Iwahori decomposition of $K_m^{\GL_l}.$ That is, $K^\circ=K_m^{\GL_l}\cap T_{\GL_l},$ $K^+=K_m^{\GL_l}\cap U_{\GL_l}$, and $K^-=K_m^{\GL_l}\cap \overline{U}_{\GL_l}.$ We let \[K_m'=\{k\in K_m^{\SO_{2l}} \ | \ j(k)\in (Q_l \cap K_m^{\SO_{2l+1}})\cup(\overline{Q}_l \cap K_m^{\SO_{2l+1}})\}.\]
We show that $j(K_m')$ almost covers $K_m^{\GL_l}$ up to elements in the unipotent radical. 

\begin{lemma}\label{lem compact gen}
    Given $k'\in K^+$ or $K^-$, respectively,  there exists $k\in K_m'$ such that $j(k)=l_l(k')v,$ where $v\in V_l\cap K_m^{\SO_{2l+1}}$ or $\overline{V}_l\cap K_m^{\SO_{2l+1}},$ respectively. Furthermore, for $a\in K_m^{\GL_{l-1}},$ we have that $l_l(\mathrm{diag}(a,1))\in K_m'$ and $j(l_l(\mathrm{diag}(a,1)))=l_l(\mathrm{diag}(a,1)).$
\end{lemma}

\begin{proof}
    Let $a\in K_m^{\GL_{l-1}}.$ Then, $\mathrm{diag}(a,1,1,a^*)\in K_m'$ and \[j(\mathrm{diag}(a,1,1,a^*))=l_l(\mathrm{diag}(a,1)).\] Next, let
    \[
    u=\left(
    \begin{matrix}
        I_{l-1} & x \\
        0 & 1
    \end{matrix}
    \right)\in K_m^{\GL_l}
    \]
    and write
    \[
    x=\begin{bmatrix}
        x_1 \\
        x_2 \\
        \vdots \\
        x_{l-1}
    \end{bmatrix}.
    \]
    Let $W=(w_{i,j})_{i,j=1}^{l-1}$ be the $(l-1)\times(l-1)$ matrix defined by $w_{i,j}=\frac{2}{\rho}x_ix_{l-j}.$
    Let $g\in\SO_{2l}$ be decomposed as above. We let $A=I_{l}, B_{1,1}=2x, B_{2,1}=W, D_{1,1}=1, D_{1,2}=\frac{1}{\gamma}{}^txJ_{l-1}, D_{2,2}=I_{l-1}$ and $0$ elsewhere. That is,
    \[
    g=\begin{bmatrix}
        I_{l-1} & & 2x & W \\
         & 1 & & \\
         & & 1 & \frac{1}{\gamma}{}^txJ_{l-1} \\
         & & & I_{l-1}
    \end{bmatrix}.
    \]
    It is verified by direct computation that $g\in\SO_{2l}$. Furthermore $j(g)\in Q_l=L_lV_l$ and examining the decomposition, we have $j(g)=l_l(u)v$ where $v\in V_l\cap K_m^{\SO_{2l+1}}.$  Since the products $au$ generate a set larger than $K^+,$ the claim follows (note that $L_l\cap K_m^{\SO_{2l+1}}$ normalizes $V_l\cap K_m^{\SO_{2l+1}}$). Next we consider ${}^tu\in K_m^{\GL_l}.$ It is straightforward to check that $j({}^tg)=l_l({}^tu)v'$ where $v'\in\overline{V}_l\cap K_m^{\SO_{2l+1}}.$ Similarly as before, the products $a{}^tu$ generate a set larger than $K^-$ and again, the claim follows.
\end{proof}

We remark that the previous lemma has an analogue in the split case (\cite[Lemma 6.12]{HL22b}) and the strategy of the proof is similar. However, the computation here is a generalization of the split case. Indeed, for simplicity consider $l=2$, i.e., $\SO_4(J_{2l,\rho})$. Let $g$ be as in the proof above but change $W$ to be $0$. Then $g\in\SO_4(J_{2l,\rho})$ if and only if $\rho\in (F^\times)^2$ or equivalently $\SO_4(J_{2l,\rho})$ is split. So, in the split case we may take $W=0$ for a simpler computation. This shortcut is implicitly used in the proof of \cite[Lemma 6.12]{HL22b}. However, we cannot take advantage of the shortcut in the quasi-split case. In particular, the proof of Lemma \ref{lem compact gen} above works for both split and quasi-split $\SO_{2l}(J_{2l,\rho}).$
}

The next theorem shows that $B_m(g,f)-B_m(g,f')+B_m^c(g,f)-B_m^c(g,f')=0$ for $m$ large enough.

\begin{thm}\label{l thm}
The condition $\mathcal{C}(l)$ implies
$$
B_m(g,f_{\tilde{w}_{l-1}})+B_m^c(g,f_{\tilde{w}_{l-1}})=0,
$$
for any $g\in\SO_{2l}$ and $m$ large enough depending only on $f$ and $f'.$
\end{thm}

\begin{proof}
Let $m$ be large enough such that Theorem \ref{l-1 thm} and Lemma \ref{6.3} hold. Let $(\tau, V_\tau)$ be an irreducible $\psi\inv$-generic representation of $\GL_{l}$ and $v\in V_\tau.$ We consider the embedding of $\SO_{2l}$ into $\SO_{2l+1}$ as in $\S$\ref{gps and reps}. Let $j:\SO_{2l}\rightarrow\SO_{2l+1}$ denote the embedding. Recall that $Q_{l}=L_{l} V_{l}$ is the standard Siegel parabolic subgroup of $\SO_{2l+1}$ with Levi subgroup $L_{l}\cong\GL_{l}.$ Also, $P_{l-1}=M_{l-1} N_{l-1}$ is the standard parabolic subgroup of $\SO_{2l}$ with Levi subgroup $M_{l-1}\cong\GL_{l-1}\times\SO_2.$ Let $\ol{V}_{{l},m}=\ol{V}_{l} \cap H_m^{\SO_{2l+1}}$ and $\ol{N}_{l-1,m}=\{u\in \ol{N}_{l-1} \, |  j(u) \in \ol{V}_{l,m}\}$ where $\ol{N}_{l-1}$ denotes the opposite unipotent radical to $N_{l-1}.$ Then $j(\ol{N}_{{l-1},m})$ is an open compact subgroup of $\ol{V}_{l}.$ Let $i$ be a positive integer with $i\geq\mathrm{max}\{m,i_0(v),i(j(N_{{l-1},m}),v)\}$ (for notation, see Lemma \ref{5.3} for $i_0(v)$ and Lemma \ref{5.4} for $i(j(N_{{l-1},m}),v)$ where we take $D=j(N_{{l-1},m})$). We consider the sections $\xi_s^{i,v}\in I(\tau,s,\psi\inv)$ and $\tilde{\xi}_{1-s}^{i,v}=M(\tau,s,\psi\inv)\xi_s^{i,v}\in I(\tau^*,1-s,\psi\inv)$ defined in $\S\ref{sections}.$

First, we compute the non-intertwined zeta integrals $\Psi(W_m^f,\xi_s^{i,v}).$ We have that $N_{l-1} M_{l-1} \ol{N}_{l-1}$ is dense in $\SO_{2l}$. Hence, we can replace the integral over $U_{\SO_{2l}}\backslash\SO_{2l}$ with an integral over $U_{\SO_{2l}}\backslash N_{l-1} M_{l-1} \ol{N}_{l-1}\cong U_{\GL_{l-1}}\backslash\GL_{l-1} \times \ol{N}_{l-1}$ where $\GL_{l-1}$ is realized in $M_{l-1}\subseteq\SO_{2l}.$ Note for $g=u_1 m_{l-1}(a) \ol{u}_2\in U_{\GL_{l-1}}\backslash\GL_{l-1} \times \ol{N}_{l-1}$ we use the quotient Haar measure $dg=|\mathrm{det}(a)|^{-{(l-1)}} d\ol{u}_2 da.$

We have
$$
\Psi(W_m^f,\xi_s^{i,v})=\int_{U_{\GL_{l-1}}\backslash\GL_{l-1}}\int_{\ol{N}_{l-1}} W_m^f(  m_{l-1}(a)\ol{u}) \xi_s^{i,v}(w_{l,l}j(m_{l-1}(a)\ol{u}), I_{l})|\mathrm{det}(a)|^{-{(l-1)}}d\ol{u}da.
$$
By Equation \eqref{xi nonintertwined}, we have
$$
\Psi(W_m^f,\xi_s^{i,v})=\int_{U_{\GL_{l-1}}\backslash\GL_{l-1}}\int_{\ol{N}_{l-1,i}} W_m^f(  m_{l-1}(a)\ol{u}) \xi_s^{i,v}(w_{l,l}j(m_{l-1}(a)\ol{u}), I_{l})|\mathrm{det}(a)|^{-{(l-1)}}d\ol{u}da.
$$
Furthermore, for $\ol{u}\in\ol{N}_{{l-1},i}$, we have $\ol{u} \in H_m$. By Proposition \ref{partialBesselprop},
$$
W^f_m(g j(\ol{u}))= W_m^f(g),
$$
for any $g\in\SO_{2l}.$ From this and Equation \eqref{xi nonintertwined}, we have
$$
\Psi(W_m^f,\xi_s^{i,v})=\mathrm{Vol}(\ol{N}_{l-1,i})\int_{U_{\GL_{l-1}}\backslash\GL_{l-1}} W_m^f(  m_{l-1}(a)) \xi_s^{i,v}(w_{l,l}j(m_{l-1}(a)), I_{l})|\mathrm{det}(a)|^{-{(l-1)}}da.
$$
By Lemma \ref{6.3}(1),
$$
W_m^f( m_{l-1}(a)) )=W_m^{f'}( m_{l-1}(a)),
$$
and hence 
$$
\Psi(W_m^f,\xi_s^{i,v})=\Psi(W_m^{f'},\xi_s^{i,v}).
$$

By assumption, $\gamma(s,\pi\times\tau,\psi)=\gamma(s,\pi'\times\tau,\psi).$ Since $\Psi(W_m^f,\xi_s^{i,v})=\Psi(W_m^{f'},\xi_s^{i,v}),$ the local functional equation gives
$$\Psi(W_m^f,\tilde{\xi}_{1-s}^{i,v})=\Psi(W_m^{f'},\tilde{\xi}_{1-s}^{i,v}).$$
By Theorem \ref{l-2 theorem}, $$
B_m(g,f)-B_m(g,f')= B_m(g,f_{\tilde{w}_{l-1}}),
$$
for any $g\in\SO_{2l}.$ $B_m(g,f_{\tilde{w}_{l-1}})$ is supported on $\Omega_{\tilde{w}_{l-1}}$ and hence
\begin{align*}
0=\Psi(W_m^f,\tilde{\xi}_{1-s}^{i,v})-\Psi(W_m^{f'},\tilde{\xi}_{1-s}^{i,v})=\int_{U_{\SO_{2l}}\backslash \Omega_{\tilde{w}_{l-1}}} B_m(g,f_{\tilde{w}_{l-1}}) \tilde{\xi}_{1-s}^{i,v}(w_{l,l}j(g), I_{l})dg.
\end{align*}

We have $U_{\SO_{2l}}\backslash \Omega_{\tilde{w}_{l-1}}=\sqcup_{w\in\mathrm{B}_{l-1}(\SO_{2l})} T_{\SO_{2l}} w U_{\SO_{2l}}.$  Let \[twu\in \sqcup_{w\in\mathrm{B}_{l-1}(\SO_{2l})} T_{\SO_{2l}} w U_{\SO_{2l}}.\] By Proposition \ref{l embed}, we have  $j(twu)\in Q_l w_l V_l$ if and only if $$t=\mathrm{diag}(t_1,\dots,t_{l-1},\left(\begin{matrix}
a & b\rho \\
b & a
\end{matrix}\right),t_{l-1}\inv,\dots,t_1\inv) \in T_{\SO_{2l}}.$$ where $a\neq 1.$ Thus, we can restrict the torus $T_{\SO_{2l}}$ to the subset
$$T_l=\{t=\mathrm{diag}(t_1,\dots,t_{l-1},\left(\begin{matrix}
a & b\rho \\
b & a
\end{matrix}\right),t_{l-1}\inv,\dots,t_1\inv)\in T_{\SO_{2l}} \, | \, a\neq \pm 1\}.$$ Note that the case $a=-1$ is excluded in this set. Indeed, while $a=1$ is excluded by Proposition \ref{l embed}, we show that it follows from Theorem \ref{l-1 thm} that the case $a=-1$ does not contribute to integral.

Assume for the moment that $a=-1.$ Then $b=0$ and $t=-I_{2l}t'$ where $t'=\mathrm{diag}(-t_1,\dots,-t_{l-1},1,1,-t_{l-1}\inv,\dots,-t_1\inv).$ Let $w\in\mathrm{B}_{l-1}(\SO_{2l}).$ Then there exists $w'\in W(\GL_{l-1})$ such that $w=t_{l-1}(w')\tilde{w}_{l-1}$. Let $u''\in U_{\SO_{2l}}.$ Then there exists $u'\in U_{\GL_{l-1}}$ and $u\in N_{l-1}$ such that $u''=t_{l-1}(u')u.$ Since $\tilde{w}_{l-1}t_{l-1}(u')=t_{l-1}((u')^*)\tilde{w}_{l-1}$, we have $twu=t't_{l-1}(A)\tilde{w}_{l-1}u$ for some $A\in\GL_{l-1}$ and $u\in N_{l-1}.$ Thus 
$B_m(twu,f_{\tilde{w}_{l-1}})=\omega_\pi(-I_{2l})B_m(t't_{l-1}(A)\tilde{w}_{l-1}u,f_{\tilde{w}_{l-1}}),$ where $\omega_\pi$ is the central character of $\pi.$ By Lemma \ref{6.3}(3), if $u\in N_{l-1}\backslash(U_m \cap N_{l-1})$, then $B_m(t't_{l-1}(A)\tilde{w}_{l-1}u,f_{\tilde{w}_{l-1}})=0.$ If $u\in U_m\cap N_{l-1}$, by Lemma \ref{partialBesselprop}, it follows that \begin{align*}
    B_m(t't_{l-1}(A)\tilde{w}_{l-1}u,f_{\tilde{w}_{l-1}})=\psi_m(u)B_m(t't_{l-1}(A)\tilde{w}_{l-1},f_{\tilde{w}_{l-1}})=0,
\end{align*} where the last equality follows from Theorem \ref{l-1 thm}. Hence, we have shown that, for any $u\in U_{SO_{2l}},$  $B_m(twu,f_{\tilde{w}_{l-1}})=0$. Therefore,
\begin{align*}
0&=\int_{U_{\SO_{2l}}\backslash \Omega_{\tilde{w}_{l-1}}} B_m(g,f_{\tilde{w}_{l-1}}) \tilde{\xi}_{1-s}^{i,v}(w_{l,l}j(g), I_{l})dg \\
&=\sum_{w\in B_{l-1}(\SO_{2l})}\int_{T_l\times U_{\SO_{2l}}} B_m(g,f_{\tilde{w}_{l-1}}) \tilde{\xi}_{1-s}^{i,v}(w_{l,l}j(g), I_{l})dtdu.
\end{align*}

For $ w\in \mathrm{B}_{l-1}(\SO_{2l})$, there exists $w'\in W(\GL_{l-1})$ such that $w=t_{l-1}(w')\tilde{w}_{l-1}.$ Then, for $t\in T_l$, it follows from  Proposition \ref{l embed} that $w_{l,l}tw=l_l(\gamma A)w_lx$ where $A\in\GL_l$ and $x\in V_l$ with
 $$A=\left(\begin{matrix}
& \frac{\gamma}{2}(1-a) \\
\mathrm{diag}(t_{l-1}\inv,\dots,t_{1}\inv)(w')^* &
\end{matrix}\right)^*.$$ This matrix $A$ will be involved in computing $\tilde{\xi}_{1-s}^{i,v}(w_{l,l}j(g), I_{l})$ in the above integral. Indeed if we let $u=I_{2l}$, then this is readily seen from Equation \eqref{xi intertwined} (more generally, $u$ will also play a nontrivial role, but its discussion is more technical). 

Note that the assignment $tw\mapsto A$ is a two-to-one map. Indeed, recall that for $$t=\mathrm{diag}(t_1,\dots,t_{l-1},\left(\begin{matrix}
a & b\rho \\
b & a
\end{matrix}\right),t_{l-1}\inv,\dots,t_1\inv)\in T_{\SO_{2l}},$$
we have
$$ctc=\mathrm{diag}(t_1,\dots,t_{l-1},\left(\begin{matrix}
a & -b\rho \\
-b & a
\end{matrix}\right),t_{l-1}\inv,\dots,t_1\inv)\in T_{\SO_{2l}}.$$
Since $t\in T_l,$ it follows that $b\neq 0$ and hence $t\neq ctc$ (in fact, $T_l$ is precisely the subset of $T_{\SO_{2l}}$ which is not fixed by conjugation by $c$). However, both $tw$ and $ctcw$ result in the same $A.$ To apply Proposition \ref{JS Prop}, we need a one-to-one map. Hence, we partition $T_l$ into two sets $T_1$ and $T_2$ with the property that $t\in T_1$ if and only if $ctc\in T_2.$ 

We have
\begin{align}\label{eqn T_1+T_2}
\begin{split}
    0&=\sum_{w\in B_{l-1}(\SO_{2l})}\int_{T_l\times U_{\SO_{2l}}} B_m(g,f_{\tilde{w}_{l-1}}) \tilde{\xi}_{1-s}^{i,v}(w_{l,l}j(g), I_{l})dtdu \\
    &=\sum_{w\in B_{l-1}(\SO_{2l})}\int_{T_1\times U_{\SO_{2l}}} B_m(g,f_{\tilde{w}_{l-1}}) \tilde{\xi}_{1-s}^{i,v}(w_{l,l}j(g), I_{l})dtdu \\
    &+\sum_{w\in B_{l-1}(\SO_{2l})}\int_{T_2\times U_{\SO_{2l}}} B_m(g,f_{\tilde{w}_{l-1}}) \tilde{\xi}_{1-s}^{i,v}(w_{l,l}j(g), I_{l})dtdu.
    \end{split}
\end{align}

Let $t\in T_2, w\in\mathrm{B}_{l-1}(\SO_{2l}),$ $u\in U_{\SO_{2l}},$ and  $t'=ctc\in T_1.$ By Lemma \ref{lemma conj bessel},
\begin{align*}B_m(twu,f_{\tilde{w}_{l-1}})=B_m(ct'cwu,f_{\tilde{w}_{l-1}})=B_m(ct'cwu,f_{\tilde{w}_{l-1}})&=B_m(c\tilde{t}\inv t'wu'\tilde{t}c,f_{\tilde{w}_{l-1}}) \\
&=B_m^c( t'wu',f_{\tilde{w}_{l-1}}),
\end{align*}
 where $u'=\tilde{t}cuc\tilde{t}\inv\in U_{\SO_{2l}}.$
 Write $u=(u_{i,j})_{i,j=1}^l.$ Note that $u_{l,l+1}=0.$ Then
 $$
u'=\left(\begin{matrix}
(u_{i,j})_{i,j=1}^{i,j=l-1} & (-u_{i,l})_{i=1}^{l-1}  & (u_{i,l+1})_{i=1}^{l-1} & * \\
 & 1 &  & * \\
 & & 1 & * \\
 & & & *
\end{matrix}\right).
$$
 Also, the embedding of $u$ in $\SO_{2l+1}$ is
$$
j(u)=\left(\begin{matrix}
(u_{i,j})_{i,j=1}^{i,j=l-1} & \left(\begin{matrix}
\frac{(u_{i,l+1})_{i=1}^{l-1}}{2} & * & * \end{matrix}\right) & * \\
 & I_3 & * \\
 & & *
\end{matrix}\right).
$$
Let $\tilde{u}=\left(\begin{matrix}
(u_{i,j})_{i,j=1}^{i,j=l-1} &
\frac{(u_{i,l+1})_{i=1}^{l-1}}{2} \\
 & 1
\end{matrix}\right).$ Then $j(u)=l_l(\tilde{u})n_3.$ where $n_3\in V_l.$ The embedding takes $twu$ to $ n_4 l_l(\gamma A \tilde{u}^*) w_l n_5$ for some $n_4, n_5\in V_l.$ Thus, by Equation \eqref{5.4}, $$\tilde{\xi}_{1-s}^{i,v}(w_{l,l} j(twu), I_l)=\mathrm{Vol}(\ol{V}_{l,i})|\mathrm{det}(\frac{1}{2}A)|^{\frac{1}{2}-s}W_v^*(\mathrm{diag}(\frac{1}{2},\dots,\frac{1}{2})A \tilde{u}^*).$$ Also, the embedding takes $t'wu'$ to $ n_4' l_l(\gamma A \tilde{u}^*) w_l n_5'$ for some $n_4', n_5'\in V_l.$ Again, by Equation \eqref{5.4}, $$\tilde{\xi}_{1-s}^{i,v}(w_{l,l} j(t'wu'), I_l)=\mathrm{Vol}(\ol{V}_{l,i})|\mathrm{det}(\frac{1}{2}A)|^{\frac{1}{2}-s}W_v^*(\mathrm{diag}(\frac{1}{2},\dots,\frac{1}{2})A \tilde{u}^*).$$
Thus it follows that
$$
\tilde{\xi}_{1-s}^{i,v}(w_{l,l} j(twu), I_l)=
\tilde{\xi}_{1-s}^{i,v}(w_{l,l} j(t'wu'), I_l).$$ Therefore, we have that
\begin{align*}
    &\sum_{w\in B_{l-1}(\SO_{2l})}\int_{T_2\times U_{\SO_{2l}}} B_m(g,f_{\tilde{w}_{l-1}}) \tilde{\xi}_{1-s}^{i,v}(w_{l,l}j(g), I_{l})dtdu \\=&\sum_{w\in B_{l-1}(\SO_{2l})}\int_{T_1\times U_{\SO_{2l}}} B_m^c(g,f_{\tilde{w}_{l-1}}) \tilde{\xi}_{1-s}^{i,v}(w_{l,l}j(g), I_{l})dtdu,
\end{align*}
 and hence from Equation \eqref{eqn T_1+T_2},
\begin{align*}
    0=&\sum_{w\in B_{l-1}(\SO_{2l})}\int_{T_1\times U_{\SO_{2l}}} (B_m(g,f_{\tilde{w}_{l-1}})+B_m^c(g,f_{\tilde{w}_{l-1}})) \tilde{\xi}_{1-s}^{i,v}(w_{l,l}j(g), I_{l})dtdu.
\end{align*}

{ Let $K_m^{\SO_{2l+1}}$ be a compact open subgroup of $\SO_{2l+1}$ for which $\tilde{\xi}_{1-s}^{i,v}(w_{l,l}j(g), I_{l})$ is right-invariant. Consider the set $K_m'$ defined before Lemma \ref{lem compact gen}. Since $B_m+B_m^c$ is smooth, we have $(B_m+B_m^c)(gk,f_{\tilde{w}_{l-1}})=(B_m+B_m^c)(g,f_{\tilde{w}_{l-1}})$ for any $k\in K_m'.$ Also, note that $K_m'\subseteq K_m^{\SO_{2l}}$ and hence has finite volume. It follows that
\begin{align*}
    0=&\sum_{w\in B_{l-1}(\SO_{2l})}\int_{T_1\times U_{\SO_{2l}}\times K_m'} (B_m(gk,f_{\tilde{w}_{l-1}})+B_m^c(gk,f_{\tilde{w}_{l-1}})) \tilde{\xi}_{1-s}^{i,v}(w_{l,l}j(gk), I_{l})dtdu.
\end{align*}}

Recall that
 \begin{align*}
     A&=\left(\begin{matrix}
& \frac{\gamma}{2}(1-a) \\
\mathrm{diag}(t_{l-1}\inv,\dots,t_{1}\inv)(w')^* &
\end{matrix}\right)^* \\
&=\left(\mathrm{diag}(\frac{\gamma}{2}(1-a),t_{l-1}\inv,\dots,t_{1}\inv)\left(\begin{matrix}
    &  1 \\
    (w')^* &
\end{matrix}\right)\right)^*.
 \end{align*}
 For such $A,$ we have $A\in T_{\GL_l}w''$ where $w''\in W(\GL_l).$
Let $Y'$ be the subset of $\GL_l$ defined by 
the image of the map $tw\mapsto A$ and set $Y:=\gamma Y'.$ The map $tw\mapsto \gamma A$ gives a bijection from $T_1$ to $Y,$ which denote by $\zeta.$ 
{ Let $p$ denote the composition of the embedding $j:T_1 B_{l-1}(\SO_{2l}) U_{\SO_{2l}} K_m'\rightarrow V_l L_l w_l V_l$ with the projection $V_l L_l w_l V_l\rightarrow L_l\cong GL_l.$ Note that $p(T_1 B_{l-1}(\SO_{2l}) U_{\SO_{2l}})\subseteq Y$ and $K_m'':=p(K_m')\subseteq K_m^{\GL_l}.$ Set $X=U_{\GL_l}Y U_{\GL_l} K_m''.$ We define a function $f:\GL_l\rightarrow\mathbb{C}$ by $f(u_1yu_2k)=\psi(u_1)(B_m(\zeta\inv(y)m_l(u^*)\tilde{w}_{l-1},f_{\tilde{w}_{l-1}})+B_m^c(\zeta\inv(y)m_{l}(u^*)\tilde{w}_{l-1},f_{\tilde{w}_{l-1}})),$ where $u_1,u_2\in U_{\GL_l},$ $y\in Y,$ and $k\in K_m'',$ and $f(x)=0$ if $x\not\in X.$ We remark that $f$ is well-defined and smooth (this follows similarly to the argument in the split case in \cite[p. 38]{HL22b}).

Then
\begin{align*}
0=&\sum_{w\in B_{l-1}(\SO_{2l})}\int_{T_1\times U_{\SO_{2l}}\times K_m'} (B_m(gk,f_{\tilde{w}_{l-1}})+B_m^c(gk,f_{\tilde{w}_{l-1}})) \tilde{\xi}_{1-s}^{i,v}(w_{l,l}j(gk), I_{l})dtdu \\
=&\int_{U_{\GL_l}\backslash \GL_l} f(A)|\mathrm{det}(A)|^{\frac{1}{2}-s}W_v^*(A) dA.
 \end{align*}
Finally, from Lemma \ref{JS Prop}, we have $f(x)=0$ for any $x\in X.$ In particular,
$$
B_m(twu,f_{\tilde{w}_{l-1}})+B_m^c(twu,f_{\tilde{w}_{l-1}})=0
$$
for any $t\in T_1$, $w\in \mathrm{B}_{l-1}(\SO_{2l}),$ and $u\in U_{\SO_{2l}}.$ Conjugation by $c$ fixes $\mathrm{B}_{l-1}(\SO_{2l})$ and sends $T_1$ to $T_2$. Thus, it follows that 
$$B_m(twu,f_{\tilde{w}_{l-1}})+B_m^c(twu,f_{\tilde{w}_{l-1}})=0
$$
for any $t\in T_2$ and $w\in \mathrm{B}_{l-1}(\SO_{2l}).$ Since $T_l=T_1\cup T_2,$ we have proven the theorem. }
\end{proof}

Similarly, we have the equality of twisted gamma factors of $\pi$ and  $\pi^c$. We again remark that this can be implied by the work of \cite{Art13, JL14}, but our proof gives an intrinsic argument.

\begin{cor}\label{conj gamma3}
Let $\pi$ be an irreducible  $\psi$-generic supercuspidal representation of $\SO_{2l}$. Then $\gamma(s, \pi\times\tau,\psi)=\gamma(s, \pi^c\times\tau,\psi)$ for all irreducible $\psi^{-1}$-generic representations $\tau$ of $\GL_{n}$ for any $n\leq l.$ 
\end{cor}

\begin{proof}
By Corollaries \ref{conj gamma1} and \ref{conj gamma2}, $\gamma(s, \pi\times\tau,\psi)=\gamma(s, \pi^c\times\tau,\psi)$ for all irreducible $\psi^{-1}$-generic representations $\tau$ of $\GL_{k}$ for $k\leq l-1.$ Repeating the steps and the change of variables performed in the proof of Theorem \ref{l thm} show that the zeta integrals for $\pi$ and $\pi^c$ are equal and hence so are the $\gamma$-factors for the twist by $\GL_l.$
\end{proof}

Finally, we can complete the proof of the converse theorem. From Theorem \ref{l thm}, for any $g\in\SO_{2l},$ we have
$$
B_m(g,f)-B_m(g,f')+B_m^c(g,f)-B_m^c(g,f')=0.$$ 
By the uniqueness of Whittaker models, we must have $\pi\cong\pi'$ or $\pi\cong \pi'^c$ (see \cite[Theorem 8.2]{HL22a} or \cite[Theorem 8.1]{Haz23a}). This completes the proof of Theorem \ref{converse thm intro}.

\bibliographystyle{amsplain}
\bibliography{QS_local_converse_even_orthog}

\end{document}